\documentclass[11pt,reqno]{amsart}
\usepackage{amsmath}
\usepackage{amsthm}
\usepackage{amssymb}
\usepackage{enumerate}
\usepackage{amscd}
\usepackage{color}
\usepackage{pb-diagram}
\usepackage{graphicx}
\usepackage[all, cmtip]{xy}
\usepackage{esint}
\usepackage{hyperref}
\hypersetup{pdftex,colorlinks=true,allcolors=blue}
\usepackage{hypcap}
\usepackage[titletoc]{appendix}
\usepackage{chngcntr}

\theoremstyle{plain}

\newtheorem{lemma}{Lemma}
\newtheorem{prop}[lemma]{Proposition}
\newtheorem{coro}[lemma]{Corollary}
\newtheorem{theorem}[lemma]{Theorem}
\newtheorem{conj}[lemma]{Conjecture}

\newtheorem*{prop*}{Proposition}
\theoremstyle{definition}
\newtheorem{defn}[lemma]{Definition}
\newtheorem{remark}[lemma]{Remark}
\newtheorem{ex}[lemma]{Example}

\theoremstyle{remark}

\numberwithin{equation}{section}
\newenvironment{enumeratei}{\begin{enumerate}[\upshape (i)]}{\end{enumerate}}

\newcommand{\p}{\partial}

\newcommand{\Hol}{\textup{Hol}}
\newcommand{\Rea}{\textup{Re}}

\begin{document}
\title[Invariant metrics]{Invariant metrics on negatively pinched complete K\"ahler manifolds}

\author{Damin Wu}
\address{Department of Mathematics\\
University of Connecticut\\
341 Mansfield Road U1009
Storrs, CT 06269-1009,
USA}

\email{damin.wu@uconn.edu}

\author{Shing--Tung Yau}
\address{Department of Mathematics \\
				Harvard University \\
				One Oxford Street, Cambridge MA 02138}
\email{yau@math.harvard.edu}
\thanks{The first author was partially supported by the NSF grant DMS-1611745. The second author was partially supported by the NSF grant DMS-1308244 and DMS-1607871. The first author would like to thank Professor H. Wu for bringing his attention to the Bergman metric in 2006.}
\maketitle

\begin{abstract} 
We prove that a complete K\"ahler manifold with holomorphic curvature bounded between two negative constants admits a unique complete K\"ahler-Einstein metric. We also show this metric and the Kobayashi-Royden metric are both uniformly equivalent to the background K\"ahler metric. Furthermore, all three metrics are shown to be uniformly equivalent to the Bergman metric, if the complete K\"ahler manifold is simply-connected, with the sectional curvature bounded between two negative constants. In particular, we confirm two conjectures of R. E. Greene and H. Wu posted in 1979.
\end{abstract}

\maketitle

\tableofcontents

\section{Introduction}
The classical Liouville's theorem tells us that the complex plane $\mathbb{C}$ has no bounded nonconstant holomorphic functions, while, by contrast, the unit disk $\mathbb{D}$ has plenty of bounded nonconstant holomorphic functions. From a geometric viewpoint, the complex plane does not admit any metric of negative curvature, while the unit disk admits a metric, the Poincar\'e metric, of negative curvature.

In a higher dimensional analogue, the unit disk is replaced by the simply-connected complete K\"ahler manifold. 
It is believed that a simply-connected complete K\"ahler manifold $M$ with sectional curvature bounded above by a negative constant has many nonconstant bounded holomorphic functions (cf. \cite[p. 678, Problem 38]{Yau:1982}). In fact, it is conjectured that such a manifold is biholomorphic to a bounded domain in~$\mathbb{C}^n$ (cf. \cite[p. 225]{Siu-Yau:1977}, \cite[p. 98]{Wu:1983}).

The negatively curved complex manifolds are naturally associated with the invariant metrics.  
An \emph{invariant metric} is a metric $L_M$ defined on a complex manifold $M$ such that every biholomorphism $F$ from $M$ to itself gives an isometry $F^*L_M = L_M$. Thus, the invariant metric depends only on the underlying complex structure of $M$.

There are four classical invariant metrics, the Bergman metric, the Carath\'eodory-Reiffen metric,  the Kobayashi-Royden metric, and the K\"ahler-Einstein metric of negative scalar curvature. It is known that on a bounded, smooth, strictly pseudoconvex domain in $\mathbb{C}^n$, all four classical invariant metrics are uniformly equivalent to each other (see, for example, \cite{Diederich:1970, Graham:1975, Cheng-Yau:1980, Lempert:1981, Beals-Fefferman-Grossman:1983, Wu:1993} and references therein). The equivalences do not extended to weakly pseudoconvex domains (see, for example, \cite{Diederich-Fornaess-Herbort:1984} and references therein for the inequivalence of the Bergman metric and the Kobayashi-Royden metric).

On K\"ahler manifolds, R. E. Greene and H. Wu have posted two remarkable conjectures concerning the uniform equivalences of the Kobayashi-Royden metric and the Bergman metric. Their first conjecture states as below.
\begin{conj}[{\cite[p. 112, Remark (2)]{Greene-Wu:1979}}] \label{con:GW-KR}
Let $(M, \omega)$ be a simply-connected complete K\"ahler manifold satisfying $- B \le \textup{sectional curvature} \le -A$ for two positive constants $A$ and $B$. Then, the Kobayashi-Royden metric $\mathfrak{K}$ satisfies
\[
    C^{-1} | \xi |_{\omega} \le \mathfrak{K}(x, \xi) \le C | \xi |_{\omega}, \quad \textup{for all $x \in M$ and $\xi \in T'_x M$}.
\]
Here $C>0$ is a constant depending only on $A$ and $B$.
\end{conj}
As pointed out in \cite[p. 112]{Greene-Wu:1979}, it is well-known that the left inequality in the conjecture follows from the Schwarz lemma and the hypothesis of sectional curvature bounded above by a negative constant (see also Lemma~\ref{le:KR-left}).

Our first result confirms this conjecture. In fact, we prove a stronger result, as we relax the sectional curvature to the holomorphic sectional curvature, and remove the assumption of simply-connectedness. 
\begin{theorem}\label{th:WY-KR}
Let $(M, \omega)$ be a complete K\"ahler manifold whose holomorphic sectional curvature $H(\omega)$ satisfies $- B \le H(\omega) \le -A$ for some positive constants $A$ and $B$. Then, the Kobayashi-Royden metric $\mathfrak{K}$ satisfies
\[
    C^{-1} | \xi |_{\omega} \le \mathfrak{K}(x, \xi) \le C | \xi |_{\omega}, \quad \textup{for all $x \in M$ and $\xi \in T'_x M$}.
\]
Here $C>0$ is a constant depending only on $A$, $B$ and $\dim M$.
\end{theorem}

Under the same condition as Theorem~\ref{th:WY-KR}, we construct a unique complete K\"ahler-Einstein metrics of negative Ricci curvature, and show that it is uniformly equivalent to the background K\"ahler metric.
\begin{theorem} \label{th:WYc1}
Let $(M, \omega)$ be a compete K\"ahler manifold whose holomorphic sectional curvature $H(\omega)$ satisfies $-\kappa_2 \le H(\omega) \le -\kappa_1$ for constants $\kappa_1, \kappa_2 > 0$.  Then $M$ admits a unique complete K\"ahler-Einstein metric $\omega_{\textup{KE}}$ with Ricci curvature equal to $-1$, satisfying
\[
   C^{-1} \omega \le \omega_{\textup{KE}} \le C \omega \quad \textup{on $M$}
\]
for some constant $C>0$ depending only on $\dim M$, $\kappa_1$ and $\kappa_2$. Furthermore, the curvature tensor $R_{\textup{m}, \textup{KE}}$ of $\omega_{\textup{KE}}$ and all its covariant derivatives are bounded; that is, for each $l \in \mathbb{N}$,
\[
    \sup_{x \in M} \big|\nabla^l R_{\textup{m}, \textup{KE}} (x) \big|_{\omega_{\textup{KE}}} \le C_l
\]
where $C_l > 0$ depends only on $l$, $\dim M$, $\kappa_1$, and $\kappa_2$.
\end{theorem}

Theorem~\ref{th:WYc1} differs from the previous work on complete noncompact K\"ahler-Einstein metrics such as \cite{Cheng-Yau:1980, Cheng-Yau:1986, Tian-Yau:1987, Wu:2008} in that we put no assumption on the sign of Ricci curvature $\textup{Ric}(\omega)$ of metric $\omega$, nor on $\textup{Ric}(\omega) - \omega$. The proof makes use of a new complex Monge-Amp\`ere type equation, which involves the K\"ahler class of $t \omega - \textup{Ric}(\omega)$ rather than that of $\omega$.
This equation is inspired by our recent work~\cite{Wu-Yau:2015}. Theorem~\ref{th:WYc1} can be viewed as a complete noncompact generalization of \cite[Theorem 2]{Wu-Yau:2015} (for its generalizations on compact manifolds, see for example \cite{Tosatti-Yang:2015, Diverio-Trapani:2016, Wu-Yau:2016:quasi, Yang-Zheng:2017}.)

We now discuss the second conjecture of Greene-Wu concerning the Bergman metric. Greene-Wu has obtained the following result, motivated by the work of the second author and Y. T. Siu~\cite{Siu-Yau:1977}.
\begin{theorem}[{\cite[p. 144, Theorem H (3)]{Greene-Wu:1979}}] \label{th:GrWuH}
Let $(M, \omega)$ be a simply-connected complete K\"ahler manifold such that $-B \le \textup{sectional curvature} \le -A$ for some positive constants $A$ and $B$. Then, $M$ possesses a complete Bergman metric $\omega_{\mathfrak{B}}$ satisfying
\[
   \omega_{\mathfrak{B}} \ge C \omega \quad \textup{on $M$},
\]
for some constant $C>0$ depending only on $\dim M$, $A$, and $B$. Moreover, the Bergman kernel form $\mathfrak{B}$ on $M$ satisfies
\begin{equation} \label{eq:qsBkf}
    A_1 \omega^n \le \mathfrak{B} \le A_2 \omega^n \quad \textup{on $M$},
\end{equation}
for some positive constants $A_1, A_2$ depending only on $\dim M$, $A$, and $B$.
\end{theorem}

It is shown in \cite[p. 144, Theorem H (2)]{Greene-Wu:1979} that, if a simply-connected complete K\"ahler manifold $M$ satisfies $-B/ r^2 \le \textup{sectional curvature} \le - A/r^2$ outside a compact subset of $M$, then $M$ possesses a complete Bergman metric, where $r$ is the distance from a fixed point. Greene-Wu proposed two conjectures concerning their Theorem H. The first conjecture is that the lower bound $-B/r^2$ in the hypothesis of Theorem~H~(2) can be removed. This has been settled by B.~Y.~Chen and J. H. Zhang~\cite{Chen-Zhang:2002}. The second conjecture is as below. 
\begin{conj}[{\cite[p. 145, Remark (3)]{Greene-Wu:1979}}] \label{con:GW-B}
The Bergman metric $\omega_{\mathfrak{B}}$ obtained in Theorem~\ref{th:GrWuH} satisfies
\[
    \omega_{\mathfrak{B}} \le C_1 \omega \quad \textup{on $M$}
\]
for some constant $C_1>0$. As a consequence, the Bergman metric $\omega_{\mathfrak{B}}$ is uniformly equivalent to the background K\"ahler metric $\omega$.
\end{conj}

 Conjecture~\ref{con:GW-B} now follows from the following result.
\begin{theorem}\label{th:WY-B}
Let $(M, \omega)$ be a complete, simply-connected, K\"ahler manifold such that $-B \le \textup{sectional curvature} \le -A<0$ for some positive constants $A$ and $B$. Then, Bergman metric $\omega_{\mathfrak{B}}$ has bounded geometry, and satisfies
\[
    \omega_{\mathfrak{B}} \le C_1 \omega \quad \textup{on $M$},
\]
where the constant $C_1>0$ depending only on $A$, $B$, and $\dim M$. As a consequence, the Bergman metric $\omega_{\mathfrak{B}}$ is uniformly equivalent to the K\"ahler metric $\omega$.
\end{theorem}
The  simply-connectedness assumption is necessary for the equivalence of $\omega_{\mathfrak{B}}$ and $\omega$. For example, let $M = \mathbb{P}^1\setminus \{0, 1, \infty\}$. Then, $M$ has a complete K\"ahler-Einstein metric with curvature equal to $-1$; however, $M$ admits no Bergman metric.
Another example is the punctured disk $\mathbb{D}^* = \mathbb{D} \setminus \{0\}$ together with the complete Poincar\'e metric $\omega_{\mathcal{P}} = \sqrt{-1} dz \wedge d \bar{z} / (|z| \log |z|^2)^2$. Note that the Bergman metric on $\mathbb{D}^*$ is $\omega_{\mathfrak{B}} = \sqrt{-1} dz \wedge d\bar{z} / (1 - |z|^2)^2$, which cannot dominate $\omega_{\mathcal{P}}$ at the origin.

The equivalence of $\omega_{\mathfrak{B}}$ and $\omega$ in Theorem~\ref{th:WY-B} has been known in several cases: For instance, when $M$ is a bounded strictly pseudoconvex domain with smooth boundary, this can be shown by using the asymptotic expansion of Monge-Amp\`ere equation (see \cite{Beals-Fefferman-Grossman:1983} for example). The second author with K. Liu and X. Sun \cite{Liu-Sun-Yau:2004} has proved the result for $M$ being the Teichm\"uller space and the moduli space of Riemann surfaces, on which they in fact show that several classical and new metrics are all uniformly equivalent (see also \cite{Yeung:2005}); compare Corollary~\ref{co:WY-3} below.

As a consequence of the above theorems, we obtain the following result on a complete, simply-connected, K\"ahler manifold with negatively pinched sectional curvature. 
\begin{coro} \label{co:WY-3}
Let $(M, \omega)$ be a complete, simply-connected, K\"ahler manifold satisfying $- B \le \textup{sectional curvature} \le -A$ for two positive constants $A$ and $B$. Then, the K\"ahler-Einstein metric $\omega_{\textup{KE}}$, the Bergman metric $\omega_{\mathfrak{B}}$, and the Kobayashi-Royden metric $\mathfrak{K}$ all exist, and are all uniformly equivalent to $\omega$ on $M$, where the equivalence constants depend only on $A$, $B$, and $\dim M$.
\end{coro}
Corollary~\ref{co:WY-3} in particular implies that the smoothly bounded weakly pseudoconvex domain~$\Omega$ constructed in \cite{Diederich-Fornaess-Herbort:1984} and \cite[p. 491]{Jarnicki-Pflug:2013}, given by
\[
   \Omega = \{(z_1, z_2, z_3) \in \mathbb{C}^3 ; \Rea z_1 + |z_1|^2 + |z_2|^{12} + |z_3|^{12} + |z_2|^4 |z_3|^2 + |z_2|^2 |z_3|^6 < 0 \},
\]
cannot admit a complete K\"ahler metric with negative pinched sectional curvature.

In this paper we provide a unifying treatment for the invariant metrics, through developing the techniques of effective quasi-bounded geometry. The quasi-bounded geometry was originally introduced to solve the Monge-Amp\`ere equation on the complete noncompact manifold with injectivity radius zero. By contrast to solving equations, the holomorphicity of quasi-coordinate map is essential for our applications to invariant metrics. It is crucial to show the radius of quasi-bounded geometry depends only on the curvature bounds. Then, a key ingredient is the \emph{pointwise interior estimate}. Several arguments, such as Lemma~\ref{le:WYdelta1}, Lemma~\ref{le:mpcIII}, Lemma~\ref{le:WY-KR}, and Corollary~\ref{co:WintdK}, may have interests of their own.

\noindent \textbf{Notation and Convention.}
We interchangeably denote a hermitian metric by tensor $g_{\omega} = \sum_{i,j} g_{i \bar{j}} dz^{i} \otimes d \bar{z}^{j}$ and its K\"ahler form $\omega = (\sqrt{-1} /2) \sum_{i,j} g_{i\bar{j}} dz^i \wedge d\bar{z}^j$. The curvature tensor $R_{\textup{m}} = \{R_{i\bar{j}k\bar{l}}\}$ of $\omega$ is given by
\[
   R_{i\bar{j}k\bar{l}} = R\Big(\frac{\p}{\p z^i}, \frac{\p}{\p \bar{z}^j}, \frac{\p}{ \p z^k}, \frac{\p}{ \p \bar{z}^l} \Big) = - \frac{\partial^2 g_{i\bar{j}}}{\partial z^k \partial \bar{z}^l}  + \sum_{p,q = 1}^n g^{p\bar{q}} \frac{\partial g_{i\bar{q}}}{\partial z^k}  \frac{\partial g_{p \bar{j}}}{\partial \bar{z}^l}. 
\]
Let $x$ be a point in $M$ and $\eta \in T'_x M$ be a unit holomorphic tangent vector at $x$. Then, holomorphic (sectional) curvature of $\omega$ at $x$ in the direction $\eta$ is
\[
   H(\omega, x, \eta) = H(x, \eta) = R(\eta, \overline{\eta}, \eta, \overline{\eta}) = \sum_{i,j,k,l} R_{i\bar{j}k\bar{l}} \eta^i \bar{\eta}^j \eta^k \bar{\eta}^l.
\] 
We abbreviate $H(\omega) \le \kappa$ (resp. $H(\omega) \ge \kappa$) for some constant $\kappa$, if $H(\omega, x, \eta) \le \kappa$ (resp. $H(\omega, x, \eta) \ge \kappa$) at every point $x$ of $M$ and for each $\eta \in T'_x M$. 
We denote
\[
   dd^c \log \omega^n = dd^c \log \det (g_{i\bar{j}}) = \frac{\sqrt{-1}}{2} \p \bar{\p} \log \det (g_{i\bar{j}}) = - \textup{Ric}(\omega)
\]
where $d^c = \sqrt{-1} (\bar{\p} - \p) / 4$.

We say that two pseudometrics $L_1$ and $L_2$ are \emph{uniformly equivalent} or \emph{quasi-isometric} on a complex manifold $M$, if there exists a constant $C>0$ such that 
\[
    C^{-1} L_1(x, \xi) \le L_2(x, \xi) \le C L_1(x, \xi) \quad \textup{for all $x \in M$, $\xi \in T'_x M$},
\]
which is often abbreviated as $C^{-1} L_1 \le L_2 \le C L_1$ on $M$.

In many estimates, we give quite explicit constants mainly to indicate their dependence on the parameters such as  $\dim M$ and the curvature bounds.

\section{Effective quasi-bounded geometry} \label{se:qbg}

The notions of \emph{bounded geometry} and \emph{quasi-bounded geometry} are introduced by the second author and S. Y. Cheng, originally to adapt the Schauder type estimates to solve the Monge-Amp\`ere type equation on complete noncompact manifolds (see, for example, \cite{Yau:1978:Ast}, \cite{Cheng-Yau:1980, Cheng-Yau:1986}, \cite{Tian-Yau:1987, Tian-Yau:1990, Tian-Yau:1991} and \cite[Appendix]{Wang-Lin:1997}).

We use the following the formulation (compare \cite[p. 580]{Tian-Yau:1990} for example). Let $(M, \omega)$ be an $n$-dimensional complete K\"ahler manifold. For a point $P \in M$, let $B_{\omega}(P; \rho)$ be the open geodesic ball centered at $P$ in $M$ of radius $\rho$; sometimes we omit the subscript $\omega$ when there is no confusion. Denote by $B_{\mathbb{C}^n}(0; r)$ the open ball centered at the origin in $\mathbb{C}^n$ of radius $r$ with respect to the standard metric $\omega_{\mathbb{C}^n}$.
\begin{defn} \label{de:qbg}
An $n$-dimensional K\"ahler manifold $(M, \omega)$ is said to have \emph{quasi-bounded geometry}, if there exist two constants $r_2 > r_1 > 0$, 
such that for each point $P$ of $M$, there is a domain $U$ in $\mathbb{C}^n$ and a nonsingular holomorphic map $\psi: U \to M$ satisfying the following properties
\begin{enumeratei}
\item $B_{\mathbb{C}^n}(0; r_1) \subset U \subset B_{\mathbb{C}^n}(0; r_2)$ and $\psi(0) = P$;
\item there exists a constant $C> 0$ depending only on $r_1, r_2, n$ such that 
\begin{equation} \label{eq:simEu}
    C^{-1} \omega_{\mathbb{C}^n} \le \psi^*\omega \le C \omega_{\mathbb{C}^n} \quad \textup{on $U$};
\end{equation}
\item for each integer $l \ge 0$, there exists a constant $A_l$ depending only on $l$, $n, r_1, r_2$ such that 
\begin{equation} \label{eq:bddgk}
   \sup_{x \in U} \bigg| \frac{\p^{|\nu| + |\mu|} g_{i\bar{j}}}{\p v^{\mu} \p \bar{v}^{\nu}} (x) \bigg| \le A_l, \quad \textup{for all $|\mu| + |\nu| \le l$}
\end{equation}
where 
$g_{i\bar{j}}$ is the component of $\psi^*\omega$ on $U$ in terms of the natural coordinates $(v^1,\ldots, v^n)$, and $\mu, \nu$ are the multiple indices with $|\mu| = \mu_1 + \cdots + \mu_n$. 
\end{enumeratei}
The map $\psi$ is called a \emph{quasi-coordinate map} and the pair $(U, \psi)$ is called a \emph{quasi-coordinate chart} of $M$. We call the positive number $r_1$ a \emph{radius of quasi-bounded geometry}. The K\"ahler manifold $(M, \omega)$ is of \emph{bounded geometry} if in addition each $\psi:U \to M$ is biholomorphic onto its image. In this case, the number $r_1$ is called \emph{radius of bounded geometry}.
\end{defn}

The following theorem is fundamental on constructing the quasi-coordinate charts. 
\begin{theorem} \label{th:WYqc}
Let $(M, \omega)$ be a complete K\"ahler manifold. 
\begin{enumerate}
\item \label{it:WYqc:1} The manifold $(M, \omega)$ has quasi-bounded geometry if and only if  for each integer $q \ge 0$, there exists a constant $C_q >0$ such that
\begin{equation} \label{eq:bddRm}
   \sup_{P \in M} |\nabla^q R_{\textup{m}}| \le C_q,
\end{equation}
where $R_{\textup{m}} = \{ R_{i\bar{j}k\bar{l}}\}$ denotes the curvature tensor of $\omega$. In this case, the radius of quasi-bounded geometry depends only on $C_0$ and $\dim M$.
\item If $(M, \omega)$ has positive injectivity radius and the curvature tensor $R_{\textup{m}}$ of $\omega$ satisfies \eqref{eq:bddRm}, then $(M, \omega)$ has bounded geometry. The radius of bounded geometry depends only on $C_0$, $\dim M$, and also the injectivity radius $r_{\omega}$ of $\omega$ unless $r_{\omega}$ is infinity.
\end{enumerate}
\end{theorem}

Part \eqref{it:WYqc:1} in Theorem~\ref{th:WYqc} is especially useful for a K\"ahler manifold with injectivity radius zero, such as the quasi-projective manifold with positive logarithmic canonical bundle (compare, for example, \cite{Yau:1978:Ast}, \cite{Tian-Yau:1987}, \cite{Wu:2008}, \cite{Guenancia-Wu:2016}).

A subtlety in the proof of Theorem~\ref{th:WYqc} is as below. It is known that one can pullback a K\"ahler structure on the tangent space at a given point via the exponential map; by applying the $L^2$ estimate of $\bar{\p}$-operator to the geodesic normal coordinates, one obtains the holomorphic coordinates in a geodesic ball of radius $r$. The subtlety is to show that the radius $r$ is independent of the given point. We have to provide a full proof for this subtlety, as it is crucial for our applications.

The proof of Theorem~\ref{th:WYqc} requires the following Lemma~\ref{le:WYdelta1}. Its weak version, Lemma~\ref{le:WYnghd}, reformulates some classical results in the literature (see, for example, \cite[pp. 247--248]{Siu-Yau:1977}, \cite[pp. 160--161]{Greene-Wu:1979}, and \cite[p. 582]{Tian-Yau:1990}) into a form, from which we can proceed further. The strong or effective version, Lemma~\ref{le:WYdelta1}, is the form we need. Lemma~\ref{le:WYdelta1} strengthens Lemma~\ref{le:WYnghd} in that the desired holomorphic coordinates are defined in geodesic balls whose radii have uniform lower bound.

\begin{lemma} \label{le:WYnghd}
 Let $(N^n, g)$ be an $n$-dimensional K\"ahler manifold, and let $B(P; \delta_0)$ be an open geodesic ball of radius $\delta_0$ centered at a point $P$ in $N$. Suppose that $B(P; \delta_0)$ is contained in a coordinate chart in $N$ with smooth, real-valued, coordinate functions $\{x^1,\ldots, x^{n}, x^{n+1}, \ldots, x^{2n} \}$.
 Assume that the following conditions hold, where each $A_j$ denotes a positive constant.
 \begin{enumeratei}
 \item \label{it:NcutP} No cut point of $P$ is contained in $B(P; \delta_0)$.
 \item \label{it:Kgb} The sectional curvature $K(g)$ of $g$ satisfies $- A_2 \le K(g) \le A_1$ on  $B(P; \delta_0)$. 
 \item \label{it:SYdbar}
 For each $j = 1,\ldots,n$,
 \[
     |\bar{\p} (x^j + \sqrt{-1}\, x^{n+j}) |_{g} (Q) \le \phi(r), \quad \textup{for all $Q \in B(P; \delta_0)$}.
 \]
 Here $r = r(Q)$ denotes the geodesic distance $d(P, Q)$, and $\phi \ge 0$ is a continuous function on $[0, +\infty)$ satisfying
 \begin{equation} \label{eq:intp2t3}
    \int_0^{1/2} \frac{\phi^2(t)}{t^3} dt < + \infty.   
 \end{equation}
 \end{enumeratei}
 Then, 
 there exists a system of holomorphic coordinates $\{v^1,\ldots, v^n\}$ defined on a smaller geodesic ball $B(P; \delta_1)$ such that 
 \begin{equation} \label{eq:vdvP}
    v^j = x^j + \sqrt{-1} \, x^{n+j}, \quad d v^j = d (x^j + \sqrt{-1} \, x^{n+j}), \quad \textup{at $P$}
\end{equation}
for all $j = 1, \ldots, n$.
\end{lemma}

\begin{proof}
Let $h$ be a real-valued smooth function on $[0, +\infty)$ and let $\omega_g$ be the K\"ahler form of $g$. 
By conditions \eqref{it:NcutP} and \eqref{it:Kgb}, we apply the Hessian Comparison Theorem (see, for example, \cite[p. 231]{Siu-Yau:1977} and \cite[p. 4, Theorem 1.1]{Schoen-Yau:1994})
to $h(r)$ to obtain
\[
   4 dd^c h(r) \ge  \min\big\{ 2h'(r) \sqrt{A_1} \cot (\sqrt{A_1}\, r), h'(r) \sqrt{A_1} \cot (\sqrt{A_1} \, r) + h''(r) \big\} \, \omega_g,
\]
for all $x \in B(P; \delta_0)$, where $r = r(x) = d(x, P)$. Letting $h(r)$ be $r^2$ and $\log (1 + r^2)$, respectively, yields
\begin{align}
   dd^c r^2 & \ge \frac{\pi}{4} \omega_g, \label{eq:ddcr2}\\
   dd^c \log (1 + r^2) & \ge \frac{4 \pi}{17} \omega_g, \notag
\end{align}
for all $x \in B(P; \delta)$, where $\delta$ is a constant satisfying
\begin{equation} \label{eq:delta04}
   0 < \delta \le \min\Big\{\delta_0, \frac{1}{4}, \frac{\pi}{4\sqrt{A_1}} \Big\}.
\end{equation}
Inequality \eqref{eq:ddcr2} in particular implies that $B(P; \delta)$ is a Stein manifold.
On the other hand, by \eqref{it:Kgb}, the Ricci curvature $\textup{Ric}(\omega_g)$ of $g$ satisfies
\[
    |\textup{Ric}(\omega_g)|_g \le \sqrt{n} |R_{\textup{m}}|_g \le \frac{34}{3} n^{3/2} (A_1 + A_2).
\]
Pick a constant $l>0$ such that
\[
    \frac{4 \pi}{17} l \ge \frac{34}{3}n^{3/2} (A_1 + A_2) + 1,
\]
and let
\[
   \varphi_1 = l \log (1 + r^2), \quad \varphi_2 = (2n + 2) \log r, \quad \varphi = \varphi_1 + \varphi_2.
\]
Then, 
\[
   dd^c \varphi_1 + \textup{Ric}(\omega_g) \ge \omega_g \quad \textup{on $B(P; \delta)$}.
\]

Let $0 \le \chi \le 1$ be a smooth function on $\mathbb{R}$ such that $\chi \equiv 1$ on the closed interval $[0, \delta/6]$ and $\chi \equiv 0$ on $[\delta/3,  +\infty)$.  
Let
\[
   w^j = x^j + \sqrt{-1} x^{n+j}, \quad j = 1, \ldots, n.
\]
It follows from \cite[Proposition 2.1, p. 244]{Siu-Yau:1977} (see also \cite[Lemma 4 and Remark, p. 208]{Mok-Siu-Yau:1981} for Stein K\"ahler manifolds) that there is a smooth function $\beta^j$ on $B(P; \delta)$ such that
\begin{equation} \label{eq:dbeta}
   \bar{\p} \beta^j = \bar{\p} [ (\chi \circ r) w^j] \quad \textup{on $B(P; \delta)$}
\end{equation}
and satisfies
\begin{equation} \label{eq:SYGWTY}
   \int_{B(P;\delta)} |\beta^j|^2 e^{-\varphi} dV_g \le \int_{B(P;\delta)} | \bar{\p} ((\chi \circ r) w^j) |_g^2 e^{-\varphi} dV_g, \quad j = 1 \ldots, n.
\end{equation}
By condition \eqref{it:SYdbar}, 
\begin{align}
   \int_{B(P; \delta/6)} | \bar{\p} ((\chi \circ r) w^j) |_g^2 e^{-\varphi} dV_g 
   & = \int_{B(P; \delta/6)} |\bar{\p} w^j|_g^2 e^{-\varphi} dV_g \notag \\
   & \le C(n, A_1, A_2) \int_0^{\delta / 6} \frac{\phi^2(r)}{r^3} d r \label{eq:L2pwj} \\
   & \le C(n, A_1, A_2) < +\infty, \notag
\end{align}
where we use the standard volume comparison $dV_g \le C(n, A_1, A_2) r^{2n-1} dr dV_{\mathbb{S}^{2n-1}}$ for $r \le \delta \le 1/4$, and $C(n, A_1, A_2) > 0$ denotes a generic constant depending only on $n$, $A_1, A_2$. This together with \eqref{eq:SYGWTY} imply
\[
   \beta^j = 0, \quad d \beta^j = 0 \quad \textup{at $P$}.
\]
Let
\begin{equation} \label{eq:vjwj}
   v^j = (\chi \circ r) w^j - \beta^j, \quad j = 1, \ldots, n.
\end{equation}
Then $v^j$ is holomorphic and satisfies \eqref{eq:vdvP} for each $j$. By the inverse function theorem, the set of functions $\{v^1,\ldots, v^n\}$ forms a holomorphic coordinate system in a smaller ball $B(P; \delta_1)$ where $0 < \delta_1 < \frac{1}{6}\min\{\,\delta_0, 1/4, \pi / (4 \sqrt{A_1}\,)\}$.
\end{proof}
\begin{remark}
Condition \eqref{eq:vdvP} in particular includes two cases, $\phi(t) = t^{1 + a}$  with constant $a > 0$, and $\phi(t) = t^k (-\log t)^{-l}$ with $k, l \ge 1$.  The former is sufficient for our current application. For clarity, we specify $\phi(t) = t^{1+a}$ in the lemma below.
\end{remark}
\begin{lemma} \label{le:WYdelta1}
Let $(N^n, g)$ and $B(P; \delta_0)$ be given as in Lemma~\ref{le:WYnghd}, satisfying conditions \eqref{it:NcutP}, \eqref{it:Kgb}, and \eqref{it:SYdbar} with $\phi (r) = A_3 r^{1+\sigma}$ for some constant $\sigma>0$. Assume, in addition, that the metric component $\{g_{ij}\}$ of $g$ with respect to $\{x^1,\ldots, x^{2n}\}$ satisfies 
 \begin{align}
     A_4^{-1} (\delta_{ij}) \le (g_{ij})(Q) & \le A_4 (\delta_{ij}), & & 1 \le i, j \le 2n, \label{eq:ugijQP}\\
     \Big| \frac{\p g_{ij} }{\p x^k} (Q) \Big| & \le A_5, && 1 \le i, j, k \le 2n, \label{eq:udgijQ}
 \end{align} 
 for all $Q \in B(P; \delta_0)$. Then, there is a holomorphic coordinate system $\{v^1,\ldots, v^n\}$ defined on a smaller geodesic ball $B(P; \delta_1)$, for which
 \begin{enumerate}[\upshape(a)]
  \item \label{it:WYdelta1:b} the radius $\delta_1$ depends only on $\delta_0$, $n$, $A_j$, $1 \le j \le 5$, and also $\sigma$ if $\sigma<1$;
   \item \label{it:WYdelta1:a} the coordinate function $v^j$ satisfies \eqref{eq:vdvP} and
 \begin{align}
    | v^j - w^j | & \le \frac{1}{2} r^{1 + \frac{\sigma_1}{2}}, \label{eq:bdd|vj|} \\
    \Big | \frac{\p v^i}{\p w^j} - \delta_{ij} \Big| & \le \frac{1}{2} r^{\frac{\sigma_1}{2}}, \quad \Big| \frac{\p v^i}{\p \overline{w}^j} \Big| \le \frac{1}{2} r^{\frac{\sigma_1}{2}}, \label{eq:pv/pw}
 \end{align}
 on $B(P; \delta_1)$ for all $1 \le i, j \le n$, where $w^j = x^j + \sqrt{-1} x^{n+j}$, $r = r(Q) = d(P, Q)$, and $\sigma_1 \equiv \min\{\sigma, 1\}$.
  \end{enumerate}
\end{lemma}
\begin{proof}
It remains to show \eqref{it:WYdelta1:b} and \eqref{eq:bdd|vj|}. 
We start from \eqref{eq:dbeta} to obtain
\[
    \bar{\p}^* \bar{\p} \beta^j = \bar{\p}^* \bar{\p} [ (\chi \circ r) w^j] \quad \textup{on $B(P; \delta)$}.
\]
where $\delta>0$ is a constant satisfying \eqref{eq:delta04} and $w^j = x^j + \sqrt{-1} x^{n+j}$.
Since $(N, g)$ is K\"ahler, the Laplace-Beltrami operator $\Delta_g$ is equal to the $\bar{\p}$-Laplacian $\Box = \bar{\p}^*\bar{\p} + \bar{\p} \bar{\p}^*$ up to a constant factor $(-2)$ (i.e., $\Delta_g = - 2\Box$\,). It follows that
\begin{equation} \label{eq:Lgbetaf}
   \Delta_g \beta^j = \Delta_g w^j \equiv f \quad \textup{on $B(P; \delta/6)$}.
\end{equation}
One can write
\[
   \Delta_g = \frac{1}{\sqrt{g}}\frac{\p}{\p x^a} \Big(g^{a b} \sqrt{g} \frac{\p }{\p x^b} \Big) 
\]
on $B(0; \delta)$ using the given single coordinate system $\{x^1,\ldots, x^n\}$, where the summation notation is used and $1 \le a, b \le 2n$. It follows that $\Delta_g$ is of the divergence form, and is uniformly elliptic by \eqref{eq:ugijQP}. 

Applying the standard interior estimate \cite[p. 210, Theorem 8.32]{Gilbarg-Trudinger:2001} to equation \eqref{eq:Lgbetaf} yields 
\[
   | d \beta^j |_{0; B(P; \delta/24)} \le C(n, A) \Big[ \delta^{-1} | \beta^j |_{0; B(P;\delta/12)} + \delta |f|_{0; B(P; \delta/12)} \Big].
\]
Here $|\cdot|_{0; U} \equiv |\cdot|_{C^0(U)}$ for a domain $U$,
and we denote by $C(n, A)$ a generic constant depending only on $n$ and $A_j$, $1 \le j \le 5$.

 To estimate the $C^0$-norm, we use the local maximum principle (\cite[Theorem 8.17, p. 194]{Gilbarg-Trudinger:2001} with $\nu = 0$) to get
\begin{equation} \label{eq:C0L2beta}
   | \beta^j |_{0; B(P; \delta/12)} \le C(n, A) \Big[ \delta^{-n} |\beta^j |_{L^2(B(P; \delta/6))} + \delta^2 |f|_{0; B(P; \delta/6)} \Big].
\end{equation}
Combining these two estimates yields
\[
   | d \beta^j |_{0; B(P; \delta/24)} 
   \le C(n, A) \Big[ \delta^{-n-1} |\beta^j |_{L^2(B(P; \delta/6))} + \delta |f|_{0; B(P;\delta/6)} \Big].
\]
To estimate the $L^2$-norm, we apply \eqref{eq:SYGWTY} and \eqref{eq:L2pwj} to obtain
\begin{align*}
   |\beta^j|^2_{L^2 (B(P; \delta/6))}
   & = \int_{B(P; \delta/6)} |\beta^j|^2 e^{-\varphi} e^{\varphi} dV_g \\
   & \le \delta^{2n+2} \int_{B(P;\delta/6)} \big| \bar{\p} ((\chi \circ r) w^j) \big|_g^2 e^{-\varphi} dV_g \\
   & \le C(n, A_1, A_2, A_3) \delta^{2n+2} \int_0^{\delta/6} r^{2\sigma-1} dr \\
   & \le C(n, A_1, A_2, A_3) \sigma^{-1} \delta^{2n+ 2 + 2\sigma}.
\end{align*}
On the other hand, it follows from \eqref{eq:ugijQP} and \eqref{eq:udgijQ} that
\[
   |f|_{0; B(P; \delta/ 6)} \le C(n, A_4, A_5).
\]
Hence,
\begin{align}
   | d \beta^j|_{0; B(P; \delta/ 24)} 
   & \le C(n, A) [\sigma^{-1} \delta^{\sigma} + \delta] \notag \\
   & \le C(n, A) \sigma_1^{-1} \delta^{\sigma_1}, \quad \sigma_1 \equiv \min \{1, \sigma\}, \label{eq:dbetaj}
\end{align}
for all $1 \le j \le n$. It follows that
\[
   |\beta^j(Q)| \le C(n, A) \sigma_1^{-1} r^{1+\sigma_1} \quad \textup{for any $Q \in B(P; \delta/24)$},
\]
where $r = d(P, Q)$.

As in \eqref{eq:vjwj} we let
\[
   v^i = w^i - \beta^i \quad \textup{on $B(P; \delta)$}.
\]
Then, for any $Q \in B(P; \delta / 24)$, 
\begin{align*}
    | dv^1 \wedge \cdots \wedge d v^n|_g (Q)  
   & \ge |dw^1 \wedge \cdots \wedge dw^n|_g(Q) -  C(n, A) \sigma_1^{-1} \delta^{\sigma_1} \\
   & \ge A_4^{-n/2} - C(n, A) \sigma_1^{-1} \delta^{\sigma_1},
\end{align*}
where we use \eqref{eq:dbetaj} and \eqref{eq:ugijQP}. Moreover,
\[
    |v^j - w^j| = |\beta^j| \le C(n, A) \sigma_1^{-1} r^{1+\sigma_1}.
\]
Fix now a constant $\delta$ satisfying \eqref{eq:delta04} and
\[
   C(n, A) \sigma_1^{-1} \delta^{\frac{\sigma_1}{2}} \le \frac{A_4^{-n/2}}{2} \le \frac{1}{2}. 
\]
Denote $\delta_1 = \delta/24$. It follows that $dv^1,\ldots, dv^n$ for an independent set at every point in $B(P; \delta_1)$; hence, $\{v^1, \ldots, v^n\}$ forms a coordinate system on $B(P; \delta_1)$ satisfying \eqref{eq:bdd|vj|}. Estimates \eqref{eq:pv/pw} follows from $|d\beta^i(\p / \p w^k)| \le |d\beta^i| |\p / \p w^k|$, \eqref{eq:ugijQP}, and \eqref{eq:dbetaj}.
\end{proof}

\emph{Proof of Theorem~\ref{th:WYqc}}.
If $(M, \omega)$ has quasi-bounded geometry, then by definition the coordinate map $\psi$ is a local biholomorphism. It then follows from \eqref{eq:bddgk} that the curvature $R_{\textup{m}}$ of $\omega$ and all its covariant derivatives are all bounded.

Conversely, if $|R_{\textup{m}}| \le C_0$ then in particular the sectional curvature $K(\omega) \le C_0$. It follows from the standard Rauch Comparison Theorem (see, for example, \cite[p. 218, Proposition 2.4]{doCarmo:1992}) that for each $P \in M$, $B_{\omega}(P; R)$ contains no conjugate points of $P$ for $R < \pi / \sqrt{C_0}$. Fix $R = \pi / (2 \sqrt{C_0})$. Then, the exponential map 
\begin{equation} \label{eq:expp}
   \exp_p : B(0; R) \subset T_{\mathbb{R},P} M \longrightarrow M
\end{equation}
is nonsingular, and hence, a local diffeomorphism. The exponential map then pulls back a K\"ahler structure on $B(0; R)$ with K\"ahler metric $\exp_P^*\omega$ so that $\exp_P$ is a locally biholomorphic isometry. In particular, every geodesic in $B(0; R)$ through the origin is a straight line. Hence, $B(0; R)$ contains no cut point of the origin.

Pick an orthonormal basis $\{e_1,\ldots, e_{2n}\}$ of $T_{\mathbb{R}, P} M$ with respect to $g \equiv \exp_P^*\omega$, such that the associated smooth coordinate functions $\{x^1, \ldots, x^{2n}\}$ on $T_{\mathbb{R},P} M$ satisfies
\[
   \bar{\p} (x^j + \sqrt{-1} \, x^{n+j}) = 0 \quad \textup{at $x = 0$},
\]
for each $j = 1, \ldots, n$.
The complex-valued function $w^j \equiv x^j + \sqrt{-1} \, x^{n+j}$ need not be holomorphic. Nevertheless, 
we have the crucial Siu-Yau's inequality: 
If the sectional curvature $K(g)$ of $g$ satisfies $-A_2 \le K(g) \le A_1$ with constant $A_1, A_2 > 0$, then
\begin{equation} \label{eq:SYpwdbar}
  |\bar{\p} w^j|_{g}  \le n^{5/2} A r^2 e^{A_2 r^2/ 6}  \quad \textup{on $B(0; R)$},
\end{equation}
where  
\[
  r = d(0, x) = |x| = \sqrt{(x^1)^2 + \cdots + (x^{2n})^2}, \quad x \in B(0; R),
 \]
and $A>0$ is a constant depending only on $A_1$ and $A_2$. In fact, inequality \eqref{eq:SYpwdbar} follows the same procedure of estimating the dual vector field as in \cite[pp. 246--247]{Siu-Yau:1977} (its local version is also observed by \cite[p. 159, (8.22)]{Greene-Wu:1979}), with two modifications given below, due to the different upper bounds for the sectional curvature. The inequality in \cite[p. 246, line 11, i.e., p. 235, Proposition (1.5)]{Siu-Yau:1977} is replaced by
\begin{equation} \label{eq:mSYineq1}
   \Big|\nabla_{\p / \p r} \nabla_{\p / \p r} \Big(r \frac{\p}{\p x^l}\Big) \Big|_g \le  n^2 A r e^{A_2 r^2/ 6}, \quad 1 \le l \le 2n,
\end{equation}
and the inequality $|X|^2 \ge \frac{1}{2}\sum_{j=1}^n(|\lambda_j|^2+|\mu_j|^2)$ in \cite[p. 247]{Siu-Yau:1977} is replaced by 
\begin{equation} \label{eq:mSYineq2}
   | X |^2 \ge \frac{1}{2} \min\Big\{1, \frac{\sin (\sqrt{A_1} r)}{ \sqrt{A_1} r} \Big\} \sum_{j=1}^{n} (|\lambda_j|^2 + |\mu_j|^2),
\end{equation}
under the curvature condition $- A_2 \le K(g) \le A_1$; both \eqref{eq:mSYineq1} and \eqref{eq:mSYineq2} follow readily from the standard comparison argument (\eqref{eq:mSYineq2} is indeed half of \eqref{eq:Wgij} below).

Let $\{g_{ij}\}$ be the components of metric $g \equiv \exp_P^*\omega$ with respect to $\{x^j\}$. If the sectional curvature satisfies $-A_2 \le K(g) \le A_1$, then again by the standard Rauch comparison theorem we obtain
\begin{align}
   A^{-1} (\delta_{ij}) & \le g_{ij}(x) \le A (\delta_{ij}), \quad 1 \le i,j \le 2n, \label{eq:Wgij}\\  
   \Big| \frac{\p g_{ij}}{\p x^k} (x) \Big| 
   & \le n^4 A r \exp \Big( \frac{A_2}{3} r^2\Big), \quad 1 \le i,j, k \le 2n, \label{eq:Wdgij} 
\end{align}
for each $x \in B(0; R)$, where $r(x) = d(0, x)$ and $A>0$ is a constant depending only on $A_1$ and $A_2$.

Thus, we can apply Lemma~\ref{le:WYdelta1} with $B(P; \delta_0) = B(0; R)$ and $\phi(r) = C(n, C_0) r^2$ to obtain a smaller ball $B(0; \delta_1)$, on which there is a holomorphic coordinate system $\{v^1,\ldots, v^n\}$ such that $v^j(0) = 0$, $d v^j (0) = d w^j(0)$, $| v^j (x) - w^j(x) |  \le \delta_1/(2\sqrt{n}\,)$, and
\begin{align}
     \Big| \frac{\p v^i}{\p w^j} (x) - \delta_{ij} \Big| \le \frac{\delta_1}{2}, \quad \Big| \frac{\p v^i}{\p \overline{w}^j}(x) \Big| \le \frac{\delta_1}{2}, \label{eq:pvjpwk}
\end{align}
 for all $x \in B(0; \delta_1)$, $1 \le i, j \le n$.
Here the radius $1/24 \ge \delta_1>0$ depends only on $n$ and $C_0$.
Since $\textbf{v} \equiv (v^1,\ldots, v^n)$ is biholomorphic from $B(0; \delta_1)$ onto its image $U$ in $\mathbb{C}^n$, the image $U$ satisfies 
\[
     B_{\mathbb{C}^n}(0; \delta_1/2) \subset U \subset B_{\mathbb{C}^n} (0; 3 \delta_1/2).
\]

It is now standard to verify that the composition $\exp_P \circ \,\textbf{v}^{-1}$ is the desired quasi-coordinate map for $P$ on $U$. Denote by $\{g_{i\bar{j}}\}$ the components of $\exp_P^*\omega$ with respect to coordinates $\{v^j\}$, by slightly abuse of notation. By \eqref{eq:Wgij} and \eqref{eq:pvjpwk}, we obtain
  \[
      C^{-1} (\delta_{ij}) \le (g_{i\bar{j}}) \le C (\delta_{ij}), 
  \]
where $C>0$ is a generic constant depending only on $C_0$ and $n$. This proves \eqref{eq:simEu}.
The estimate of first order term $| \p g_{i\bar{j}} / \p v^k|$ follows from \eqref{eq:Wdgij} and \eqref{eq:pvjpwk}. 
The higher order estimate \eqref{eq:bddgk} follows from applying the standard Schauder estimate to the Ricci and scalar curvature equations \cite[p. 582]{Tian-Yau:1990} (see also \cite[p. 259, Theorem 6.1]{DeTurck-Kazdan:1981}).

For the second statement, fix a positive number $0 < R < r_{\omega}$, where $r_{\omega}$ denotes the injectivity radius of $(M, \omega)$. Then, for every $P\in M$, the exponential map given by \eqref{eq:expp}, i.e., $\exp_P: B(0; R) \subset T_{\mathbb{R},P} M \to M$, is a diffeomorphism onto its image. From here the same process implies $(M, \omega)$ has bounded geometry.
 \qed 

\section{Wan-Xiong Shi's Lemmas}

The following lemma is useful to construct the quasi-bounded geometry.  
\begin{lemma} \label{le:Shi}
Let $(M, \omega)$ be an $n$-dimensional complete noncompact K\"ahler manifold such that 
\begin{equation} \label{eq:pinH}
   -\kappa_2 \le H(\omega) \le - \kappa_1 < 0
\end{equation}
for two constants $\kappa_1, \kappa_2 > 0$. Then, there exists another K\"ahler metric $\tilde{\omega}$ such that satisfying
\begin{align}
  & C^{-1} \omega \le \tilde{\omega} \le C \omega, \label{eq:qsismt} \\
  & - \tilde{\kappa}_2 \le H(\tilde{\omega}) \le - \tilde{\kappa}_1 < 0, \label{eq:nHolmtR} \\
  & \sup_{x \in M} |\tilde{\nabla}^q \tilde{R}_{\alpha\bar{\beta}\gamma\bar{\sigma}}| \le C_q, \label{eq:bddmtR} 
\end{align}
where $\tilde{\nabla}^q \tilde{R}_{\textup{m}}$ denotes the $q$th covariant derivative of the curvature tensor $\tilde{R}_{\textup{m}}$ of $\tilde{\omega}$ with respect to $\tilde{\omega}$, and the positive constants $C = C(n)$, $\tilde{\kappa}_j = \tilde{\kappa}_j (n, \kappa_1, \kappa_2)$, $j = 1, 2$, $C_q = C_q(n, q, \kappa_1, \kappa_2)$ depend only on the parameters in their parentheses. 
\end{lemma}
Lemma~\ref{le:Shi} \eqref{eq:qsismt} and \eqref{eq:bddmtR} are contained in W. X. Shi~\cite{Shi:1997}. We provide below the details for the pinching estimate \eqref{eq:nHolmtR} of the holomorphic sectional curvature. Of course, if the manifold were compact, then \eqref{eq:nHolmtR} would follow trivially from the usual uniform continuity of a continuous function. However, this does not hold for a general bounded smooth function on a complete noncompact manifold. Here the maximum principle (Lemma~\ref{le:mpcIII} in Appendix~\ref{se:appMp}) has to be used.

In this section and Appendix~\ref{se:appMp}, we adopt the following convention: We denote by $\omega = (\sqrt{-1}/2) g_{\alpha \bar{\beta}} dz^{\alpha} \wedge d\bar{z}^{\beta}$ the K\"ahler form of a hermitian metric $g_{\omega}$. The real part of the hermitian metric $g_{\omega} = g_{\alpha\bar{\beta}} dz^{\alpha} \otimes d\bar{z}^{\beta}$ induces a Riemannian metric $g = g_{ij} dx^i \otimes dx^j$ on $T_{\mathbb{R}} M$ which is compatible with the complex structure $J$. 
Extend $g$ linearly over $\mathbb{C}$ to $T_{\mathbb{R}} M \otimes_{\mathbb{R}} \mathbb{C} = T'M \oplus \overline{T'M}$, and then restricting it to $T'M$ recovers $(1/2) g_{\omega}$; that is,
\[
   g(v, w) = \Rea ( g_{\omega} (\eta, \xi) ), \quad g_{\omega} (\eta, \xi) = 2 g (\eta, \overline{\xi}).
\]
Here $v, w$ are real tangent vectors, and $\eta, \xi$ are their corresponding holomorphic tangent vectors under the $\mathbb{R}$-linear isomorphism $T_{\mathbb{R}} M \to T'M$, i.e., $\eta = \frac{1}{2} (v - \sqrt{-1} Jv)$, $\xi = \frac{1}{2} (w - \sqrt{-1} Jw)$. Then, the curvature tensor $R_{\textup{m}}$ satisfies
\[
    R(\eta, \overline{\eta}, \xi, \overline{\xi}) = \frac{1}{2} R (v, Jv, Jw, w).
\]
It follows that
\[
    H(x, \eta) = R(\eta, \overline{\eta}, \eta, \overline{\eta}) = \frac{1}{2} R(v, Jv, Jv, v).
\]
Unless otherwise indicated, the Greek letters such as $\alpha, \beta$ are used denote the holomorphic vectors $\p / \p z^{\alpha}, \p / \p z^{\beta}$ and range over $\{1, \ldots, n\}$, while the latin indices such as $i, j, k$ are used to denote real vectors $\p / \p x^i, \p / \p x^j$ and range over $\{1,\ldots, 2n\}$.

\noindent \emph{Proof of Lemma~\ref{le:Shi}}.
The assumption \eqref{eq:pinH} on $H$ implies the the curvature tensor $R_{\textup{m}}$ is bounded; more precisely, 
\[
   \sup_{x \in M} |R_{\textup{m}}(x)| \le  \frac{34}{3} n^2 (\kappa_2 - \kappa_1).
\]
Here and in many places of the proof, the constant in an estimate is given in certain explicit form, mainly to indicate its dependence on the parameters such as $\kappa_i$ and $n$.
Applying \cite[p. 99, Corollary 2.2]{Shi:1997} yields that the equation
\[
   \left\{ 
   \begin{aligned}
   \frac{\p}{\p t} g_{ij}(x, t) & = - 2 R_{ij} (x, t)\\
   g_{ij}(0, t) & = g_{ij}(x)
   \end{aligned} \right.
\]
admits a smooth solution $\{g_{ij}(x,t)\} > 0$ for $0 \le t \le \theta_0(n)/(\kappa_2 - \kappa_1)$, where $\theta_0(n) > 0$ is a constant depending only on $n$. Furthermore, the curvature $R_{\textup{m}}(x,t) = \{R_{ijkl}(x,t)\}$ of $\{g_{ij}(x,t)\}$ satisfies that, for each nonnegative integer $q$,
\begin{equation} \label{eq:bdddqR}
   \sup_{x \in M} |\nabla^q R_{\textup{m}} (x,t) |^2 \le \frac{C(q, n)(\kappa_2 - \kappa_1)^2}{ t^q}, \quad \textup{for all $0 < t \le \frac{\theta_0(n)}{\kappa_2-\kappa_1}\equiv T$},
\end{equation}
where $C(q,n)>0$ is a constant depending only $q$ and $n$. 
In particular, the metric $g_{ij}(x,t)$ satisfies Assumption A in \cite[p. 120]{Shi:1997}. Then, by \cite[p. 129, Theorem 5.1]{Shi:1997}, the metric $g_{ij}(x,t)$ is K\"ahler,  and satisfies
\[
    \left\{ 
      \begin{aligned}
      \frac{\p}{\p t} g_{\alpha \bar{\beta}}(x, t) & = - 4 R_{\alpha \bar{\beta}}(x, t) \\
      g_{\alpha \bar{\beta}}(x, 0) & = g_{\alpha \bar{\beta}}(x),
      \end{aligned}
      \right.
\]
for all $0 \le t \le T$.
It follows that
\begin{equation} \label{eq:eqisg}
   e^{- t C(n) (\kappa_2 - \kappa_1)} g_{\alpha \bar{\beta}}(x) \le g_{\alpha \bar{\beta}}(x, t) \le e^{t C(n) (\kappa_2 - \kappa_1)} g_{\alpha \bar{\beta}}(x),
\end{equation}
for all $0 \le t \le T = \theta_0(n) / (\kappa_2 - \kappa_1)$. Here and below, we denote by $C(n)$ and $C_j(n)$ generic positive constants depending only on $n$. Then, for an arbitrary $0 < t \le T$, the metric $\omega(x,t) = (\sqrt{-1}/2)g_{\alpha \bar{\beta}}(x, t) dz^{\alpha} \wedge d\bar{z}^{\beta}$ satisfies \eqref{eq:qsismt} and \eqref{eq:bddmtR}; in particular, the constant $C$ in \eqref{eq:qsismt} depends only on $n$, since $tC(n)(\kappa_2-\kappa_1) \le \theta_0(n)C(n)$.

We next to show that there exists a small $0 < t_0 \le T$ so that $\omega(x,t)$ also satisfies \eqref{eq:nHolmtR} whenever $0 < t \le t_0$.  
Recall that the curvature tensor satisfies the evolution equation (see, for example, \cite[p. 143, (122)]{Shi:1997})
\begin{align*}
    \frac{\p}{\p t} R_{\alpha \bar{\beta} \gamma \bar{\sigma}} 
   & = 4 \Delta R_{\alpha \bar{\beta} \gamma \bar{\sigma}}
     + 4 g^{\mu \bar{\nu}} g^{\rho \bar{\tau}} (
    R_{\alpha \bar{\beta} \mu \bar{\tau}} R_{\gamma \bar{\sigma} \rho \bar{\nu}} + R_{\alpha \bar{\sigma} \mu \bar{\tau}} R_{\gamma \bar{\beta} \rho \bar{\nu}} - R_{\alpha \bar{\nu} \gamma \bar{\tau}} R_{\mu \bar{\beta} \rho \bar{\sigma}}) \\
    & \quad  - 2 g^{\mu \bar{\nu}} (R_{\alpha \bar{\nu}} R_{\mu \bar{\beta} \rho \bar{\tau}} + R_{\mu \bar{\beta}} R_{\alpha \bar{\nu} \rho \bar{\tau}} 
    + R_{\gamma \bar{\nu}} R_{\alpha \bar{\beta} \mu \bar{\sigma}} 
    + R_{\mu \bar{\sigma}} R_{\alpha \bar{\beta} \rho \bar{\nu}}),
\end{align*}
where $\Delta \equiv \Delta_{\omega(x,t)} = \frac{1}{2} g^{\alpha \bar{\beta}}(x, t)(\nabla_{\bar{\beta}} \nabla_{\alpha} + \nabla_{\alpha} \nabla_{\bar{\beta}})$.
It follows that
\begin{align}
   & \Big(\frac{\p}{\p t} R_{\alpha \bar{\beta} \gamma \bar{\sigma}}\Big) \eta^{\alpha} \bar{\eta}^{\beta} \eta^{\gamma} \bar{\eta}^{\sigma} \notag \\
   & \le 4 \Big(\Delta R_{\alpha \bar{\beta} \gamma \bar{\sigma}}\Big)\eta^{\alpha} \bar{\eta}^{\beta} \eta^{\gamma} \bar{\eta}^{\sigma} 
    + C_1(n) | \eta |_{g_{\alpha \bar{\beta}}(x,t)}^4 |R_{\textup{m}} (x, t)|_{\omega(x,t)}^2 \notag \\
   & \le 4 \Big(\Delta R_{\alpha \bar{\beta} \gamma \bar{\sigma}}\Big)\eta^{\alpha} \bar{\eta}^{\beta} \eta^{\gamma} \bar{\eta}^{\sigma} 
    + C_1(n) (\kappa_2 - \kappa_1)^2 | \eta |_{\omega(x,t)}^4. \label{eq:dR/dt<LRk}
\end{align}
by \eqref{eq:bdddqR} with $q = 0$. Let
\[
   H(x, \eta, t) 
   = \frac{R_{\alpha \bar{\beta} \gamma \bar{\sigma}} \eta^{\alpha} \bar{\eta}^{\beta} \eta^{\gamma} \bar{\eta}^{\sigma}}{|\eta|_{\omega(x,t)}^4}.
\]
Then, by \eqref{eq:pinH} and \eqref{eq:bdddqR},
\begin{align*}
   H(x, \eta, 0) 
   & \le - \kappa_1 , \\
   | H (x, \eta, t)| 
   & \le  |R_{\textup{m}}(x, t)|_{\omega(x,t)}  \le C_0(n) ( \kappa_2 - \kappa_1).
\end{align*}
To apply the maximum principle (Lemma~\ref{le:mpcIII} in Appendix~\ref{se:appMp}), we denote
\begin{align*}
   h(x, t)
    = \max \{ H (x, \eta, t) ; |\eta|_{\omega (x,t)} = 1\},
\end{align*}
for all $x \in M$ and $0 \le t \le \theta_0(n) / (\kappa_2 - \kappa_1)$.
Then, $h$ with \eqref{eq:dR/dt<LRk} satisfy the three conditions in Lemma~\ref{le:mpcIII}. 
It follows that 
 \[
   h(x, t) \le C_2(n) (\kappa_2 - \kappa_1)^2 t - \kappa_1.
 \]
 where $C_2(n) = C_1(n) + 8 \sqrt{n} C_0(n)^2 > 0$. 
Let
\[
    t_0 = \min\Big\{\frac{\kappa_1}{2 C_2(n)(\kappa_2 - \kappa_1)^2}, \frac{\theta_0(n)}{\kappa_2 - \kappa_1} \Big\} > 0.
\]
Then, for all $0 < t \le t_0$,
\[
   H (x, \eta, t) \le h(x, t) \le - \frac{\kappa_1}{2} < 0.
\]
Since the curvature tensor is bounded (by \eqref{eq:bdddqR} with $q=0$), we have
\[
   H (x, \eta, t) \ge - C_0(n) (\kappa_2 - \kappa_1).
\]
Thus, for an arbitrary $t \in (0, t_0]$, the metric $\omega(x, t) = (\sqrt{-1}/2) g_{\alpha \bar{\beta}}(x, t) dz^{\alpha} \wedge d\bar{z}^{\beta}$ is a desired metric satisfying \eqref{eq:nHolmtR}, and also \eqref{eq:qsismt} and \eqref{eq:bddmtR}.
\qed

\begin{lemma} \label{le:Kpinch}
Let $(M^n, \omega)$ be a complete noncompact K\"ahler manifold whose Riemannian sectional curvature is pinched between two negative constants, i.e.,
\begin{equation*} 
   - \kappa_2 \le K(\omega) \le - \kappa_1 < 0.
\end{equation*}
Then, there exists another K\"ahler metric $\tilde{\omega}$ satisfying
\begin{align*}
  & C^{-1} \omega \le \tilde{\omega} \le C \omega, \\ 
  & - \tilde{\kappa}_2 \le K(\tilde{\omega}) \le - \tilde{\kappa}_1 < 0, \\ 
  & \sup_{x \in M} |\tilde{\nabla}^q \tilde{R}_{\alpha\bar{\beta}\gamma\bar{\sigma}}| \le C_q, 
\end{align*}
where $\tilde{\nabla}^q \tilde{R}_{\alpha \bar{\beta} \gamma \bar{\sigma}}$ denotes the  $q$th covariant derivatives of $\{\tilde{R}_{\alpha \bar{\beta} \gamma \bar{\sigma}}\}$ with respect to $\tilde{\omega}$, and the positive constants $C = C(n)$, $\tilde{\kappa}_j = \tilde{\kappa}_j (n, \kappa_1, \kappa_2)$, $j = 1, 2$, $C_q = C_q(n, q, \kappa_1, \kappa_2)$ depend only on the parameters inside their parentheses. 
\end{lemma}
The proof of Lemma~\ref{le:Kpinch} is entirely similar to that of Lemma~\ref{le:Shi}, with the following modification: The function $\varphi$ is now given by
\[
   \varphi(x, v, w, t)  = R_{ijkl}(x, t) v^i w^j w^k v^l,
\]
for any $x \in M$ and $v, w \in T_{\mathbb{R},x} M$, and
\begin{align*}
   h(x, t) & = \max \{ \varphi(x, \eta, \xi, t) ; | \eta \wedge \xi |_{g(x,t)} =  1 \} \\
             & = \max \{ \varphi(x, \eta, \xi, t); |\eta|_{g(x,t)} = |\xi|_{g(x,t)} = 1, \langle \eta, \xi \rangle_{g(x,t)} = 0\}.
\end{align*}
Here $|\eta \wedge \xi|^2 = |\eta|^2 |\xi|^2 - \langle \eta, \xi \rangle^2$. The result then follows from Lemma~\ref{le:mpcK}.

\begin{appendices}

\section{Maximum principles} \label{se:appMp}
The proof of Lemma~\ref{le:Shi} uses the following maximum principle, which extends \cite[p. 124, Lemma 4.7]{Shi:1997} to tensors; compare \cite[pp. 145--147]{Shi:1997}, \cite[Theorem 9.1]{Hamilton:1982}, and \cite[pp. 139--140]{CpRF:2003}, for example.

Let $(M, \tilde{\omega})$ be an $n$-dimensional complete noncompact K\"ahler manifold. Suppose for some constant $T>0$ there is a smooth solution $\omega(x,t)>0$ for the evolution equation
\begin{equation} \label{eq:WXSA1}
   \left\{ \begin{aligned}
     \frac{\p}{\p t} g_{\alpha \bar{\beta}}(x, t) & = -  4 R_{\alpha \bar{\beta}}(x,t), &&  \textup{on $M \times [0, T]$}, \\
     g_{\alpha \bar{\beta}}(x,0) & = \tilde{g}_{\alpha \bar{\beta}}(x), && \quad x \in M,
   \end{aligned} \right.
\end{equation}
where $g_{\alpha\bar{\beta}}(x,t)$ and $\tilde{g}_{\alpha \bar{\beta}}$ are the metric components of $\omega(x,t)$ and $\tilde{\omega}$, respectively. Assume that the curvature $R_{\textup{m}} (x,t)= \{R_{\alpha \bar{\beta} \gamma \bar{\sigma}}(x,t)\}$ of $\omega(x,t)$ 
satisfies
\begin{equation} \label{eq:WXSA2}
    \sup_{M \times [0, T]} |R_{\textup{m}}(x,t)|^2 \le k_0
\end{equation}
for some constant $k_0 > 0$.
\begin{lemma} \label{le:mpcIII}
With the above assumption, suppose a smooth tensor $\{W_{\alpha \bar{\beta} \gamma \bar{\sigma}}(x, t)\}$ on $M$ with complex conjugation $\overline{W}_{\alpha\bar{\beta} \gamma \bar{\sigma}} (x,t) = W_{\beta \bar{\alpha} \sigma \bar{\gamma}}(x,t)$ satisfies
\begin{equation} \label{eq:dR/dt<LR}
   \Big(\frac{\p}{\p t} W_{\alpha \bar{\beta} \gamma \bar{\sigma}}\Big) \eta^{\alpha} \bar{\eta}^{\beta} \eta^{\gamma} \bar{\eta}^{\sigma}
   \le \big(\Delta W_{\alpha \bar{\beta} \gamma \bar{\sigma}}\big)\eta^{\alpha} \bar{\eta}^{\beta} \eta^{\gamma} \bar{\eta}^{\sigma} + C_1 |\eta|^4_{\omega(x, t)},
\end{equation}
for all $x\in M$, $\eta \in T'_x M$, $0 \le t \le T$, where $\Delta \equiv 2 g^{\alpha \bar{\beta}}(x,t)(\nabla_{\bar{\beta}} \nabla_{\alpha} + \nabla_{\alpha} \nabla_{\bar{\beta}})$ and $C_1$ is a constant. Let
\[
   h(x, t) = \max \Big\{W_{\alpha \bar{\beta} \gamma \bar{\sigma}} \eta^{\alpha} \bar{\eta}^{\beta} \eta^{\gamma} \bar{\eta}^{\sigma}; \eta \in T'_x M,  |\eta|_{\omega(x, t)} = 1 \Big\},
\]
for all $x \in M$ and $0 \le t \le T$. Suppose
\begin{align}
    \sup_{x \in M, 0 \le t \le T} | h(x, t)| & \le C_0, \label{eq:|h|C0}\\
    \sup_{x \in M} h(x, 0) & \le - \kappa, \label{eq:suph(x,0)}
\end{align}
for some constants $C_0 > 0$ and $\kappa$. Then,
\[
   h(x, t) \le (8 C_0 \sqrt{n k_0}  + C_1) t - \kappa.
\]
for all $x \in M$, $0 \le t \le T$.
\end{lemma}
\begin{proof}
We prove by contradiction. Denote
\begin{equation} \label{eq:defC}
   C = 8 C_0 \sqrt{n k_0} + C_1 > 0.
\end{equation}
Suppose 
\begin{equation} \label{eq:hx1t1}
   h(x_1, t_1) - C t_1 + \kappa > 0
\end{equation}
for some $(x_1, t_1) \in M \times [0, T]$. Then, by \eqref{eq:suph(x,0)} we have $t_1 > 0$.

Under the above conditions \eqref{eq:WXSA1} and \eqref{eq:WXSA2}, by \cite[p. 124, Lemma 4.6]{Shi:1997}, there exists a function $\theta(x, t) \in C^{\infty}(M \times [0, T])$ satisfying that
 \begin{align}
   0 < \theta (x, t) & \le 1, && \textup{on $M \times [0, T]$}, \label{eq:WXS1}\\
   \frac{\p \theta}{\p t}  - \Delta_{\omega(x, t)} \theta + 2 \theta^{-1}  |\nabla \theta|_{\omega(x,t)}^2 & \le - \theta, &&\textup{on $M \times [0, T]$},  \label{eq:WXS2}\\
   \frac{C_2^{-1}}{1 + d_0 (x_0, x)} \le \theta(x, t) & \le \frac{C_2}{1 + d_0 (x_0, x)}, && \textup{on $M \times [0, T]$}. \label{eq:WXS3}
\end{align}
where $x_0$ is a fixed point in $M$, $d_0(x,y)$ is the geodesic distance between $x$ and $y$ with respect to $\omega(x,0)$, and $C_2>0$ is a constant depends only on $n$, $k_0$, and $T$. 

Let
\[
    m_0 = \sup_{x \in M, 0 \le t \le T} \Big( \big[ h(x,t) - C t  + \kappa \big] \theta (x, t) \Big).
\]
Then, $0 < m_0 \le C_0 + |\kappa|$, by \eqref{eq:hx1t1} and \eqref{eq:WXS1}. Denote
\[
   \Lambda = \frac{2 C_2 (C_0 + C T + |\kappa|)}{m_0} > 0. 
\]
Then, for any $x \in M$ with $d_0 (x, x_0) \ge \Lambda$, 
\[
   \Big| ( h(x, t) - C t + \kappa) \theta (x, t) \Big| \le \frac{C_2 (C_0 + C T + |\kappa|)}{1 + d_0 (x, x_0)} \le \frac{m_0}{2}.
\]
It follows that the function $(h - C t + \kappa) \theta$ must attain its supremum $m_0$ on the compact set $\overline{B(x_0; \Lambda)} \times [0, T]$, where $\overline{B(x_0; r)}$ denotes the closure of the geodesic ball with respect to $\omega(x,0)$ centered at $x_0$ of radius $r$. 
Let 
\[
    f(x, \eta, t) = \frac{W_{\alpha \bar{\beta} \gamma \bar{\sigma}}  \eta^{\alpha} \bar{\eta}^{\beta} \eta^{\gamma} \bar{\eta}^{\sigma}}{|\eta|_{\omega(x,t)}^4} - C t + \kappa,
\]
for all $(x, t) \in M \times [0, T]$, $\eta \in T'_x M \setminus \{0\}$.
Then, there exists a point $(x_*, \eta_*, t_*)$ with $x_* \in \overline{B(x_0; \Lambda)}$, $0 \le t_* \le T$, $\eta_* \in T'_{x_*} M$ and $|\eta_*|_{\omega(x_*, t_*)} = 1$, such that 
\[
    m_0 = f(x_*, \eta_*, t_*) \theta (x_*, t_*) = \max_{\mathcal{S}_t \times [0, T]} (f \theta),
\]
and $t_* > 0$ by \eqref{eq:suph(x,0)}, where $\mathcal{S}_t = \{(x, \eta) \in T' M ; x \in M, \eta \in T'_x M, |\eta|_{\omega(x,t)} = 1\}$.

We now employ a standard process to extend $\eta_*$ to a smooth vector field, denoted by $\eta$ with slightly abuse of notation, in a neighborhood of $(x_*, t_*)$ in $M \times [0, T]$ such that $\eta$ is nowhere vanishing on the neighborhood, and 
\begin{equation} \label{eq:deta*}
   \frac{\p}{\p t} \eta = 0, \quad \nabla \eta = 0, \quad \Delta \eta =  0, \quad \textup{at $(x_*, t_*)$}.
\end{equation}
This extension can be done, for example, by parallel transporting $\eta_*$ from $x_*$ to each point $y$ in a small geodesic ball centered at $x_*$, with respect to metric $\omega(\cdot, t_*)$, along the unique minimal geodesic joining $x_*$ to $y$; this extension is made independent of $t$ and so $\p \eta / \p t \equiv 0$ in the geodesic ball. 

Since $f(x, \eta(x), t)$ is smooth in a neighborhood of $(x_*, t_*)$, we can differentiate $f$ and evaluate the derivatives at the point $(x_*, t_*)$ to obtain
\begin{align*}
    \frac{\p}{\p t} f 
    & = \Big( \frac{\p}{\p t} W_{\alpha \bar{\beta} \gamma \bar{\sigma}} \Big) \eta^{\alpha} \bar{\eta}^{\beta} \eta^{\gamma} \bar{\eta}^{\sigma} + 8 \Big(W_{\alpha \bar{\beta} \gamma \bar{\sigma}} \eta^{\alpha} \bar{\eta}^{\beta} \eta^{\gamma} \bar{\eta}^{\sigma}\Big)  \Big( R_{\alpha \bar{\beta}} \eta^{\alpha} \bar{\eta}^{\beta}\Big)  - C \\
    & \le \Big( \frac{\p}{\p t} W_{\alpha \bar{\beta} \gamma \bar{\sigma}} \Big) \eta^{\alpha} \bar{\eta}^{\beta} \eta^{\gamma} \bar{\eta}^{\sigma} + 8 C_0 \sqrt{n k_0} - C  \quad \textup{\big(by \eqref{eq:|h|C0} and \eqref{eq:WXSA2}\big)}\\
    & \le  \Delta f + C_1 +  8 C_0 \sqrt{n k_0} - C \quad \textup{\big(by \eqref{eq:dR/dt<LR} and \eqref{eq:deta*}\big)}\\
    & \le \Delta f, \quad \textup{by \eqref{eq:defC}}.
\end{align*}
Since $f \theta = f(x, \eta(x), t) \theta (x, t)$ attains its maximum at $(x_*, t_*)$, we have
\begin{equation} \label{eq:dftheta*}
   \frac{\p}{\p t} (f \theta) \ge 0, \quad \nabla (f\theta) = 0, \quad \Delta (f\theta) \le 0, \quad \textup{at $(x_*, t_*)$}.
\end{equation}
It follows that, at the point $(x_*, t_*)$,
\begin{align*}
  0 \le \frac{\p}{\p t} (f \theta) 
  & = \theta \frac{\p}{\p t} f + f \frac{\p}{\p t} \theta \\
  & \le \theta \Delta f  + f \frac{\p}{\p t} \theta \\
  & = \Delta (f \theta) - 2 \theta^{-1} \nabla \theta \cdot \nabla (f\theta)  + f \Big[ \frac{\p \theta}{\p t} - \Delta \theta + 2 \theta^{-1} |\nabla \theta|^2 \Big] \\
  & \le - f \theta \quad \textup{\big(by \eqref{eq:dftheta*} and \eqref{eq:WXS2}\big)} \\
  & = - m_0 < 0.
\end{align*}
This yields a contradiction. The proof is therefore completed.
\end{proof}
In the proof of Lemma~\ref{le:Shi}, we apply Lemma~\ref{le:mpcIII} with $W_{\alpha \bar{\beta} \gamma \bar{\sigma}} = R_{\alpha \bar{\beta} \gamma \bar{\sigma}}$ to estimate the holomorphic sectional curvature. For the Riemannian sectional curvature, we apply the similar result given below with $W_{ijkl} = R_{ijkl}$.
\begin{lemma} \label{le:mpcK}
Assume \eqref{eq:WXSA1} and \eqref{eq:WXSA2}. Suppose a smooth real tensor $\{W_{ijkl}(x, t)\}$ on $M$  satisfies
\[
   \Big(\frac{\p}{\p t} W_{ijkl}\Big) v^i w^j w^k v^l
   \le \big(\Delta W_{ijkl}\big) v^i w^j w^k v^l + C_1 |v|^2_{\omega(x, t)} |w|^2_{\omega(x,t)},
\]
for all $x \in M$, $v, w \in T_{\mathbb{R}, x} M$, where $\Delta \equiv g^{ij}(x,t) \nabla_i \nabla_j$ and $C_1>0$ is a constant.
Let
\[
   k(x, t) = \max \Big\{W_{ijkl} v^i w^j w^k v^l; v, w \in T_{\mathbb{R}, x} M, |v \wedge w|_{g(x, t)} = 1 \Big\},
\]
for all $x \in M$ and $0 \le t \le T$. Suppose
\begin{align*}
    \sup_{x \in M, 0 \le t \le T} | k(x, t)| & \le C_0, \\
    \sup_{x \in M} k(x, 0) \le - \kappa, 
\end{align*}
for some constants $C_0 > 0$ and $\kappa$. Then,
\[
   k(x, t) \le (8 C_0 \sqrt{n k_0} + C_1) t - \kappa.
\]
for all $x \in M$, $0 \le t \le T$. \qed
\end{lemma}

\end{appendices}

\section{Kobayashi-Royden metric and holomorphic curvature}
The \emph{Kobayashi-Royden pseudometric}, denoted by $\mathfrak{K}$, is the infinitesimal form of the Kobayashi pseudodistance. Let us first recall the definition (see, for example, \cite{Royden:1970} or \cite[Section 3.5]{Kobayashi:1998}). 

Let $M$ be a complex manifold and $T'M$ be its holomorphic tangent bundle. Define $\mathfrak{K}_M: T'M \to [0, +\infty)$ as below: For any $(x, \xi) \in T'M$, 
\[
   \mathfrak{K}_M (x, \xi) = \inf_{R>0} \frac{1}{R},
\]
where $R$ ranges over all positive numbers for which there is a $\phi \in \Hol (\mathbb{D}_R, M)$ with $\phi(0) = x$ and $\phi_*(\p / \p z |_{z =0}) \equiv d\phi(\p / \p z |_{z =0}) =  \xi$. 
Here $\Hol(X, Y)$ denotes the set of holomorphic maps from $X$ to $Y$, and 
\[
   \mathbb{D}_R \equiv \{z \in \mathbb{C}; |z| < R\}, \quad \textup{and $\mathbb{D} \equiv \mathbb{D}_1$}.
\]
Equivalently, one can verify that (cf. \cite[p. 82]{Greene-Wu:1979}), for each $(x, \xi) \in T'M$,
\begin{align}
   \mathfrak{K}_M(x, \xi) 
   & = \inf \{ |V|_{\mathcal{P}} ; V \in T' \mathbb{D}, \; \textup{there is $f \in \Hol(\mathbb{D},M)$ with $f_*(V) = \xi$} \} \notag \\
   \begin{split} 
   & = \inf \{ |V|_{\mathbb{C}}; V \in T'_0 \mathbb{D}, \; \textup{there is $f \in \Hol(\mathbb{D}, M)$ such that} \\
   & \qquad \qquad \quad f(0) = x, f_*(V) = \xi \}. 
   \end{split} \label{eq:GWF_M} 
\end{align}
Here $|\cdot |_{\mathcal{P}}$ and $|\cdot |_{\mathbb{C}}$ are, respectively, the norms with respect to the Poincar\'e metric $\omega_{\mathcal{P}} = \sqrt{-1}(1-|z|^2)^{-2} dz \wedge d\bar{z}$ and Euclidean metric $\omega_{\mathbb{C}} = \sqrt{-1} dz \wedge d\bar{z}$.

The following decreasing property of $\mathfrak{K}_M$ follows immediately from definition.
\begin{prop}[{\cite[Proposition 1]{Royden:1970}}] \label{pr:FMde}
Let $M$ and $N$ be complex manifolds and $\Psi: M \to N$ be a holomorphic map. Then,
\[
   (\Psi^* \mathfrak{K}_N) (x, \xi) \equiv \mathfrak{K}_N (\Psi(x), \Psi_*(\xi)) \le \mathfrak{K}_M (x, \xi)
\]
for all $(x, \xi) \in T'M$. In particular, if $\Psi: M \to N$ is biholomorphism then the equality holds; if $M$ is a complex submanifold of $N$ then
\[
   \mathfrak{K}_N (x, \xi) \le \mathfrak{K}_M (x, \xi).
\]
\end{prop}

\begin{ex} \label{ex:F_Dm}
Let $M$ be the open ball $B(r) = \{z \in \mathbb{C}^n ; |z| < r\}$. Then,
\begin{equation} \label{eq:FMBr}
   \mathfrak{K}_{B(r)}(a, \xi) = \bigg[ \frac{|\xi|_{\mathbb{C}^n}^2}{r^2 - |a|^2} + \frac{|\xi \cdot a|_{\mathbb{C}^n}^2}{(r^2 - |a|^2)^2}\bigg]^{1/2}, 
\end{equation}
for all $a \in \mathbb{B}_r^n$ and $\xi \in T'_a \mathbb{B}^n_r$; see \cite[p. 43, Corollary 2.3.5]{Jarnicki-Pflug:2013} for example. \qed
\end{ex}
The result below is well-known. We include a proof here for completeness.
\begin{lemma} \label{le:KR-left}
Let $(M, \omega)$ be a hermitian manifold such that the holomorphic sectional curvature $H(\omega) \le - \kappa < 0$. Then,
\[
    \mathfrak{K}_M(x, \xi) \ge \sqrt{\frac{2}{\kappa}} \, | \xi |_{\omega} \quad \textup{for each $x \in M$, $\xi \in T'_xM$}.
\]
%In particular, the pseudometric $\mathfrak{K}_M$ is a metric on $M$.
\end{lemma}
\begin{proof}
Let $\psi \in \Hol(\mathbb{D}, M)$ such that $\psi(0) = x$ and $\psi_*(v) = \xi$.
It follows from the second author's Schwarz Lemma~\cite[p. 201, Theorem $2'$]{Yau:1978:schwarz} that
\[
    \psi^*\omega \le \frac{1}{\kappa} \omega_{\mathcal{P}} \quad \textup{on $\mathbb{D}$},
\]
where $\omega_{\mathcal{P}} = (\sqrt{-1}/2) 2 (1-|z|^2)^{-2} dz \wedge d\bar{z}$. It follows that
\begin{align*}
   |\xi|^2_{\omega} 
    = \omega(x; \xi) = (\psi^*\omega)(0; v) 
    \le \frac{1}{\kappa} \omega_{\mathcal{P}}(0; v) = \frac{2}{\kappa} |v|^2_{\mathbb{C}}.
\end{align*}
Hence, $|v|_{\mathbb{C}} \ge \sqrt{2/\kappa} \, |\xi|_{\omega}$. By \eqref{eq:GWF_M}, we obtain $\mathfrak{K}_M(x, \xi) \ge \sqrt{2/\kappa} \, |\xi|_{\omega}$.
\end{proof}
The quasi-bounded geometry is essential in the following estimate.
\begin{lemma} \label{le:WY-KR}
Suppose a complete K\"ahler manifold $(M, \omega)$ has quasi-bounded geometry. Then, the Kobayashi-Royden pseudometric $\mathfrak{K}$ satisfies
\[
   \mathfrak{K}_M(x, \xi) \le C | \xi |_{\omega}, \quad \textup{for all $x \in M, \xi \in T'_x M$},
\]
where $C$ depends only on the radius of quasi-bounded geometry of $(M, \omega)$.
\end{lemma}
\begin{proof}
Let $(\psi, B (r))$ be the quasi-coordinate chart of $(M, \omega)$ centered at $x$; that is, $B(r) = \{ z \in \mathbb{C}^n; |z|< r\}$ and $\psi: B(r) \to M$ is nonsingular holomorphic map such that $\psi(0) = x$. Denote $U = \psi(B(r))$. 
Then, by Proposition~\ref{pr:FMde}, 
\[
   \mathfrak{K}_M(x, \xi) \le \mathfrak{K}_U(x, \xi) = \mathfrak{K}_{U}(\psi (0), \psi_*(v)) \le \mathfrak{K}_{B(r)} (0, v),
\]
where $v \in T'_0 (B(r))$ such that $\phi_*(v) = \xi$. It follows from \eqref{eq:FMBr} that
\[
   \mathfrak{K}_M (x, \xi) \le \mathfrak{K}_{B(r)} (0, v) = \frac{| v |_{\mathbb{C}^n}}{r}.
\]
By virtue of the quasi-bounded geometry of $(M, \omega)$, more precisely, \eqref{eq:simEu}, we have
\[
   C^{-1} | v |^2_{\mathbb{C}^n} \le (\psi^*\omega) (0; v) = \omega(x; \xi) \equiv | \xi |^2_{\omega} \le C |v|^2_{\mathbb{C}^n},
\]
where $C>0$ is a constant depending only on $r$. Hence,
\[
   \mathfrak{K}_M (x, \xi) \le \frac{\sqrt{C}}{r} | \xi |_{\omega}.
\]
This completes the proof.
\end{proof}

\noindent \emph{Proof of Theorem~\ref{th:WY-KR}.} Since $-B \le H(\omega) \le - A$, we can assume $(M, \omega)$ has quasi-bounded geometry, by Lemma~\ref{le:Shi} and Theorem~\ref{th:WYqc}. Then, the radius of quasi-bounded geometry depends only on $A$, $B$, and $\dim M$. The desired result then follows from Lemma~\ref{le:WY-KR} and Lemma~\ref{le:KR-left}. \qed

\section{Bergman metric and sectional curvature}

Let $M$ be an $n$-dimensional complex manifold. We follow \cite[Section 8]{Greene-Wu:1979} for some notations. Let $\Lambda^{(n, 0)}(M) \equiv A^{n,0}(M)$ be the space of smooth complex differential $(n, 0)$ forms on $M$.  For $\varphi, \psi \in \Lambda^{(n, 0)}$, define
\begin{equation} \label{eq:L2(n,0)}
   \langle \varphi, \psi \rangle = (-1)^{n^2/2} \int_M \varphi \wedge \overline{\psi}
\end{equation}
and 
\[
   \| \varphi \| = \sqrt{ \langle \varphi, \varphi \rangle }.
\]
Let $L^2_{(n,0)}$ be the completion of
\[
    \{ \varphi \in \Lambda^{(n,0)}; \| \varphi \| < +\infty\}
\]
with respect to $\| \cdot \|$. Then $L^2_{(n,0)}$ is a separable Hilbert space with the inner product $\langle \cdot, \cdot \rangle$. Define
\[
  \mathcal{H} = \{ \varphi \in L^2_{(n,0)} \mid \textup{$\varphi$ is holomorphic} \}.
\]
Suppose $\mathcal{H} \ne \{0\}$. Let $\{e_j\}_{j \ge 0}$ be an orthonormal basis of $\mathcal{H}$ with respect to the inner product $\langle \cdot, \cdot \rangle$. Then, the $2n$ form defined on $M \times M$ given by
 \[
    \mathfrak{B}(x, y) = \sum_{j \ge 0} e_j (x) \wedge \overline{e}_j(y), \qquad x, y \in M,
 \]
 is the Bergman kernel of $M$. The convergence of this series is uniform on every compact subset of $M \times M$ (see also Lemma~\ref{le:insupB} below). The definition of $\mathfrak{B}(x, y)$ is independent of the choice of the orthonormal basis of $\mathcal{H}$. Let
\[
    \mathfrak{B}(x) = \mathfrak{B}(x, x) =  \sum_{j \ge 0} e_j (x) \wedge \overline{e}_j(x) \qquad \textup{for all $x \in M$}.
 \]
 Then $\mathfrak{B}(x)$ is a smooth $(n,n)$-form on $M$, which is called the \emph{Bergman kernel form} of $M$. Suppose for some point $P \in M$, $\mathfrak{B}(P) \ne 0$. Define
 \[
    dd^c \log \mathfrak{B} = dd^c \log b
 \]
 where we write $\mathfrak{B}(z) =  b \, d z^1 \wedge \cdots \wedge d z^n \wedge d\bar{z}^1 \wedge \cdots \wedge d \bar{z}^n$ in terms of local coordinates $(z^1,\ldots, z^n)$. It is readily to check that this definition is well-defined. If $dd^c \log \mathfrak{B}$ is everywhere positive on $M$, then we call $dd^c \log \mathfrak{B} \equiv \omega_{\mathfrak{B}}$ the \emph{Bergman metric} on $M$.

We would like to prove Theorem~\ref{th:WY-B}. We shall use the notion of bounded geometry, together with the following results, specifically Corollary~\ref{co:WintdK}. In fact, we only need the case $\Omega$ being a \emph{bounded} domain in $\mathbb{C}^n$. Lemma~\ref{le:insupB} and Lemma~\ref{le:dzdwb} may have interests of their own. In the following, when the boundary $\p \Omega$ of $\Omega$ is empty, i.e., $\Omega = \mathbb{C}^n$, we set $\textup{dist}(E, \p \Omega) = +\infty$.
\begin{lemma} \label{le:insupB}
Let $\Omega$ be a domain in $\mathbb{C}^n$. Let $\{f_j\}_{j \ge 0}$ be a sequence of holomorphic functions on $\Omega$ satisfying the following property: 
 There is a integer $N_0 \ge 0$ such that, for all $N \ge N_0$, 
\begin{equation} \label{eq:fj<1}
   \int_{\Omega} \Big| \sum_{j=0}^N c_j f_j (z) \Big|^2 d V \le \sum_{j=0}^N |c_j|^2 \quad \textup{for all $c_j \in \mathbb{C}$, $0 \le j \le N$.}
\end{equation}
Then, the series
\[
	\sum_{j=0}^{\infty} f_j(z) \overline{f_j(w)} \equiv b(z, w)
\]
converges uniformly and absolutely on every compact subset of $\Omega \times \Omega$. Furthermore,
for every compact subset $E$ of $\Omega$,
\begin{equation} \label{eq:c0b}
    \max_{z, w \in E} |b(z, w)| \le \frac{C(n)}{\textup{dist}(E, \p \Omega)^{2n}},
\end{equation}
where $C(n)>0$ is a constant depending only on $n$.
\end{lemma}
\begin{proof}
First, suppose that $\p \Omega$ is nonempty.
We assert that, for any $z \in \Omega$,
\begin{equation} \label{eq:C0bN}
     \sum_{j=0}^N | f_j(z)|^2 \le \frac{C(n)}{\textup{dist}(z, \p \Omega)^{2n}} \quad \textup{for all $N \ge N_0$}.
\end{equation}
Here and below, we denote by $C(n) > 0$ a generic constant depend only on $n$. Assume \eqref{eq:C0bN} momentarily. By the Cauchy-Schwarz inequality,
\begin{align*}
   \sum_{j=0}^N \Big| f_j(z) \overline{f_j(w)}\Big| 
    & \le \sqrt{\sum_{j=0}^N | f_j(z)|^2 } \, \sqrt{\sum_{j=0}^N | f_j(w) |^2}  \\
    & \le  \frac{C(n)}{\textup{dist}(z, \p \Omega)^{n} \cdot \textup{dist}(w, \p \Omega)^n} \\
    & \le \frac{C(n)}{\textup{dist}(E, \p\Omega)^{2n}}, 
\end{align*}
for all $z, w$ in the given compact subset $E$ and for all $N \ge N_0$. Then, letting $N \to +\infty$ yields \eqref{eq:c0b}. 

To show the first statement, by the Cauchy-Schwarz inequality it is sufficient to show that the uniform convergence of $\sum_{j=0}^{\infty} |f_j(z)|^2$ on every compact subset $E$ of $\Omega$. (This is not an immediate consequence of \eqref{eq:C0bN}, however) Let us denote by $B(z; r)$ the open ball in $\mathbb{C}^n$ centered at $z$ of radius $r$.
Let $\delta = \textup{dist}(E, \p \Omega)/4>0$. Then, for each $z_0 \in E$, $B(z_0; 2\delta) \subset \Omega$. By \eqref{eq:C0bN},
\[
    \sum_{j=1}^{\infty}  \int_{B(z_0; \delta)}  | f_j(z)|^2 dV \le \frac{C(n)}{\delta^{2n}} \textup{Vol}(B(z_0; \delta)) \le C(n) < +\infty.
\]
It follows that, for each $\varepsilon>0$, there exists a constant $L$, depending only on $\varepsilon$, such that
\[
    \sum_{j = l}^{l + m} \int_{B(z_0; \delta)} | f_j(z)|^2 dV < \varepsilon, \quad \textup{for all $l \ge L$, and $m \ge 1$}.
\]
On the other hand, applying the mean value inequality to subharmonic function $\sum_{j=l}^{l+m} |f_j(z)|^2$ on $B(z_0; \delta)$ yields
\begin{align*}
    \sum_{j=l}^{l+m} | f_j (z)|^2 \le \frac{C(n)}{\delta^{2n}} \int_{B(z_0; \delta)} \sum_{j=l}^{l+m} |f_j(w)|^2 dV_w.
\end{align*}
Hence,
\[
   \sup_{B(z_0; \delta)} \sum_{j=l}^{l+m} | f_j(z)|^2 \le \frac{C(n)}{\delta^{2n}} \varepsilon.
\]
Since $E$ can be covered by finitely many balls such as $B(z_0; \delta)$, we have proven the uniform convergence of $\sum |f_j(z)|^2$ on $E$.

To show \eqref{eq:C0bN}, fix an arbitrary $z \in \Omega$ and $N \ge N_0$. We can assume, without loss of generality, that $|f_0(z)|^2 + |f_1(z)|^2 + \cdots + |f_N(z)|^2 \ne 0$. Denote
\[
    \epsilon = \frac{\textup{dist}(z, \p \Omega)}{4}  > 0.
\]
Applying the mean value inequality to the subharmonic function $|\sum_{j=0}^N c_j f_j(z) |^2$  yields
\begin{align*}
   \bigg|\sum_{j=0}^N c_j f_j(z) \bigg|^2
   &  \le \frac{1}{\textup{Vol}(B(z; \epsilon))} \int_{B(z; \epsilon)} \Big|\sum_{j=0}^N c_j f_j (\zeta) \Big|^2 d \zeta \\
   & \le \frac{C(n)}{\epsilon^{2n}} \int_{\Omega} \Big|\sum_{j=0}^N c_j f_j (\zeta) \Big|^2 d \zeta \\
   & \le \frac{C(n)}{\textup{dist}(z, \p \Omega)^{2n}} \sum_{j=0}^N |c_j|^2, \qquad \textup{by \eqref{eq:fj<1}}.
\end{align*} 
Letting 
\[
   c_j = \frac{\overline{f_j(z)}}{(|f_0(z)|^2 + \cdots + |f_N (z)|^2)^{1/2}}, \quad 0 \le j \le N,
\]
yields \eqref{eq:C0bN}. This proves the result for $\Omega$ with nonempty boundary. 

If $\p \Omega$ is empty, then $\Omega = \mathbb{C}^n$. We can replace $\Omega$ in the previous proof by a large open ball $B(0; R)$ which contains the compact subset $E$. The same process yields
\[
   \max_{z, w \in E} |b(z, w)| \le \frac{C(n)}{\textup{dist}(E, \p B(0; R))^{2n}} \to 0, \quad \textup{as $R \to +\infty$}.
 \]
 This shows \eqref{eq:c0b}, and hence, $b \equiv 0$, for the case $\textup{dist}(E, \p \Omega) = +\infty$.
\end{proof}
\begin{remark}
An example of $b(z, w)$ in Lemma~\ref{le:insupB} is the classical Bergman kernel function, for which the \emph{equality} in \eqref{eq:fj<1} holds for \emph{all} $N \ge 0$ and $c_j$, $0 \le j \le N$. The arguments are well-known and standard (compare, for example, \cite[p. 121--122]{Bochner-Martin:1948}). For our applications on manifold, however, we have to state and derive the estimate under the weaker inequality hypothesis \eqref{eq:fj<1}, and our estimate constant needs to be explicit on $\textup{dist}(E, \p \Omega)$. 
\end{remark}
\begin{lemma} \label{le:dzdwb}
  Let $\Omega$ be a domain in $\mathbb{C}^n$. Let $b(z, w)$ be a continuous function which is holomorphic in $z$ and $\overline{w}$, and satisfies $\overline{b(z, w)} = b(w, z)$, for all $z, w \in \Omega$.
  If $\Omega \ne \mathbb{C}^n$, then, for each compact subset $E \subset \Omega$, 
\begin{equation} \label{eq:dzdwb}
   \big| \p_z^{\alpha} \p_{\overline{w}}^{\beta} b(z, w) \big| \le \frac{C(n) \alpha! \beta!}{\textup{dist}(E, \p \Omega)^{|\alpha| + |\beta|}} \max_{x, y \in E_{\Omega}} |b(x, y)|, \quad \textup{for all $z, w \in E$}.
\end{equation}
Here $C(n)>0$ is a constant depending only on $n$, $\alpha$ and $\beta$ are multi-indices with $\alpha! \equiv \alpha_1!\cdots \alpha_n!$, $|\alpha| \equiv \alpha_1 + \cdots + \alpha_n$, $\p_z^{\alpha} \equiv (\p_{z^1})^{\alpha_1} \cdots (\p_{z^n})^{\alpha_n}$, and 
\[
   E_{\Omega} = \{z \in \Omega ; \textup{dist}(z, E) \le \textup{dist}(E, \p \Omega)/ 2 \}.
\]
If $\Omega = \mathbb{C}^n$ then \eqref{eq:dzdwb} continues to hold, with $E_{\Omega}$ replaced by any closed ball whose interior contains $E$.
\end{lemma}
\begin{proof}
It is sufficient to show \eqref{eq:dzdwb} for the case $\Omega \subsetneq \mathbb{C}^n$; the case $\Omega = \mathbb{C}^n$ follows similarly.
The inequality clearly holds when $\alpha = \beta = 0$, since $E$ is contained in $E_{\Omega}$. Consider the case $\beta = 0$ but $\alpha \ne 0$. Let $\delta = \frac{1}{4\sqrt{n}} \textup{dist}(E, \p \Omega) > 0$. Pick $z, w \in E$. By Cauchy's integral formula,
\begin{align*}
    \p_z^{\alpha} b(z, w)  = & \frac{\alpha_1! \cdots \alpha_n!}{(2\pi \sqrt{-1})^n}  \int_{\{|\zeta^1-z^1= \delta |\}} \cdots \\
   &\int_{\{|\zeta^n - z^n| = \delta\}} \frac{b(\zeta, w)\, d\zeta^1 \cdots d\zeta^n}{(\zeta^1 - z^1)^{\alpha_1+1}\cdots(\zeta^n - z^n)^{\alpha_n+1}} 
\end{align*}
It follows that
\begin{align}
   |\p_z^{\alpha} b(z, w)|
   & \le \frac{C(n) \alpha!  }{\delta^{ | \alpha |} } \sup_{\zeta \in \mathbb{D}^n(z; \delta)} | b (\zeta, w)| \label{eq:intdK1D} \\
   & \le \frac{C(n)  \alpha! }{\delta^{|\alpha|}} \max_{\zeta, y  \in E_{\Omega}} | b(\zeta, y)|.  \notag
\end{align}
Here $\mathbb{D}^n(z; \delta) \equiv \{ \zeta \in \mathbb{C}^n; |\zeta^1 - z^1| < \delta, \ldots, |\zeta^n - z^n| < \delta\}$ satisfies $\mathbb{D}^n(z; \delta) \subset E_{\Omega}$.

Consider the general case $\alpha, \beta \ne 0$. Applying \eqref{eq:intdK1D} with $b(z, w)$ replaced by $\p_{\overline{w}}^{\beta} b(z, w)$ yields
\begin{align*}
   |\p_z^{\alpha} \p_{\overline{w}}^{\beta} b(z, w)| 
   & \le  \frac{C(n) \alpha!}{\delta^{|\alpha|}} \sup_{\zeta \in \mathbb{D}^n (z; \delta)} |\p_{\overline{w}}^{\beta} b(\zeta, w) | && \\
   & = \frac{C(n)\alpha!}{\delta^{|\alpha|}} \sup_{\zeta \in \mathbb{D}^n (z; \delta)} |\p_{w}^{\beta} b(w, \zeta) | && \textup{\Big(since $\overline{b(\zeta, w)} = b(w, \zeta)$\Big)} \\
   & \le \frac{\alpha! \beta! C(n)}{\delta^{|\alpha| + |\beta|}} \max_{x, y \in E_{\Omega}} |b(x, y)|, && \textup{by \eqref{eq:intdK1D}}.
\end{align*}
Here $C(n)>0$ denotes a generic constant depending only on $n$.
\end{proof}
\begin{coro} \label{co:WintdK}
Let $\Omega$ be a domain in $\mathbb{C}^n$, and let $b(z, w)$ be the function given in Lemma~\ref{le:insupB}. For each compact subset $E \subset \Omega$, 
\[
   \big| \p_z^{\alpha} \p_{\overline{w}}^{\beta} b(z, w) \big| \le \frac{\alpha! \beta! C(n)}{\textup{dist}(E, \p \Omega)^{2n + |\alpha| + |\beta|}}, \quad \textup{for all $z, w \in E$}.
\]
Here $C(n)>0$ is a constant depending only on $n$, and $\alpha, \beta \in (\mathbb{Z}_{\ge 0})^n$ are multi-indices, $\p_z^{\alpha} \equiv (\p_{z^1})^{\alpha_1} \cdots (\p_{z^n})^{\alpha_n}$, $\alpha! \equiv \alpha_1!\cdots \alpha_n!$, and $|\alpha| \equiv \alpha_1 + \cdots + \alpha_n$.
\end{coro}
\begin{remark}
Corollary~\ref{co:WintdK} in particular implies a pointwise \emph{interior estimate} for the Bergman kernel. This may be compared with the  global estimates of Bergman kernel function in the smooth bounded domain satisfying certain boundary condition such as Bell's Condition R (see, for example, \cite{Kerzman:1972}, \cite{Bell-Boas:1981}, and \cite[p. 144]{Chen-Shaw:2001} and references therein). Those estimates are based on the pseudo-local estimate of the $\bar{\p}$-Neumann operator. The method here is entirely elementary, without assuming any boundary condition.
\end{remark}

\begin{lemma} \label{le:WBeM}
Let $(M^n, \omega)$ be a complete, simply-connected, K\"ahler manifold satisfy
\begin{equation} \label{eq:Kk1k2}
    - \kappa_2 \le K(\omega) \le - \kappa_1 < 0
\end{equation}
for two positive constants $\kappa_2 > \kappa_1 > 0$.
Let $\mathfrak{B}(z, z)$ and $\omega_{\mathfrak{B}}$ be the Bergman kernel form and Bergman metric on $M$. Assume that  $\mathfrak{B} /\omega^n \ge c_0$ on $M$ for some constant $c_0>0$. Then, $\omega_{\mathfrak{B}}$ has bounded geometry, and satisfies
\begin{equation} \label{eq:conGWB}
   \omega_{\mathfrak{B}} \le C(n, c_0, \kappa_1, \kappa_2) \, \omega \quad \textup{on $M$},
\end{equation}
where $C(n, c_0, \kappa_1, \kappa_2)>0$ is a constant depending only on $n$, $c_0$, $\kappa_1$, and $\kappa_2$.
\end{lemma}
\begin{proof}
By \eqref{eq:Kk1k2} and Lemma~\ref{le:Kpinch} we can assume, without loss of generality, that the curvature tensor of $\omega$ and all its covariant derivatives are bounded. On the other hand, it follows from \eqref{eq:Kk1k2} and the standard Cartan-Hadamard theorem that, for a point $P \in M$, the exponential map $\exp_P : T_{\mathbb{R}, P} M \to M$ is a diffeomorphism. This, in particular, implies that the injectivity radius of $M$ is infinity. Thus, the manifold $(M, \omega)$ is of bounded geometry, by the second statement of Theorem~\ref{th:WYqc}.

 Since $\omega$ has bounded geometry, there exists a constant $r>0$, depending only on $n$, $\kappa_1$, $\kappa_2$, such that for each point $p \in E$, there is a biholomorphism $\psi_p$ from the open ball $B(r) \equiv B_{\mathbb{C}^n}(0; r)$ onto its image in $M$ such that $\psi_p(0) = p$ and $\psi_p^*(\omega)$ is uniformly equivalent to Euclidean metric on $B(r)$ up to infinite order. In particular, let $g_{i\bar{j}}$ be the metric component of $\psi_p^*(\omega)$ with respect to holomorphic coordinates $v^1, \ldots, v^n$ centered at $p$; then
  \begin{equation} \label{eq:bgC}
      C^{-1} (\delta_{ij}) \le (g_{i\bar{j}}) \le C (\delta_{ij}) 
  \end{equation}
  on $B(r)$. Here $B(r)$ denotes a ball in $\mathbb{C}^n$ centered at the origin of radius $r>0$, and $C>0$ is a generic constant depending only on $\kappa_1$, $\kappa_2$, and $n$.
  
Let $\{\phi_j\}_{j \ge 0}$ be an orthonormal basis of the Hilbert space $\mathcal{H}$ with respect to the inner product $\langle \cdot, \cdot \rangle$ given in \eqref{eq:L2(n,0)}. Then, by definition
 \[
    \mathfrak{B}(P, Q) = \sum_{j \ge 0} \phi_j (P) \wedge \overline{\phi_j(Q)}, \qquad \textup{for all $P, Q \in M$}.
 \]
Write $\phi_j = f_j(v) dv^1 \wedge \cdots \wedge dv^n$ in the chart $(B(r), \psi_p, v^j)$, for which we mean, as a standard convention, $\psi_p^*\phi_j (v) = f_j(v) dv^1 \wedge \cdots \wedge dv^n$ for $v \in B(r)$. 
Then, each $f_j$ is holomorphic on $B(r)$, $j \ge 0$.

 We \textbf{claim} that the domain $B(r)$ and sequence $\{f_j\}$ satisfy the requirement, \eqref{eq:fj<1}, in Lemma~\ref{le:insupB}. Indeed, for each $\phi \in \mathcal{H}$ with $\psi_p^* \phi = h \, dv^1\wedge \cdots \wedge dv^n$ on $B(r)$, we have
\begin{align*}
   \int_{B(r)} | h(v) |^2 dV 
   & = (-1)^{n^2/2} \int_{B(r)} | h (v)|^2 dv^1 \wedge \cdots \wedge dv^n \wedge d\bar{v}^1 \wedge \cdots \wedge d\bar{v}^n \\
   & =  (-1)^{n^2/2} \int_{\psi_p(B(r))} \phi \wedge \overline{\phi} \\
   & \le (-1)^{n^2/2} \int_M \phi \wedge \overline{\phi} \equiv \langle \phi, \phi \rangle.
\end{align*}
Now for any $N \ge 0$ and any $c_j \in \mathbb{C}$, $0 \le j \le N$, substituting
\[
   \phi = \sum_{j=0}^N c_j \phi_j \quad \textup{with} \quad \textup{$h (v) = \sum_{j=0}^N c_j f_j$ in $B(r)$}
\]
yields
\[
   \int_{B(r)} \Big| \sum_{j=0}^N c_j f_j(v) \Big|^2 dV \le \langle \phi, \phi \rangle = \sum_{j=0}^N |c_j|^2.
\]
This verifies \eqref{eq:fj<1}; hence, the claim is proved. Therefore, we have
\[
    \psi_p^* \mathfrak{B} (v, w)  = b (v, w) dv^1\wedge \cdots \wedge dv^n \wedge d \overline{w}^1 \wedge \cdots \wedge d \overline{w}^n, \quad v, w \in B(r),
\]
in which
\[
   b(v, w) = \sum_{j \ge 0} f_j (v) \overline{f_j(w)}
\]
is a continuous function in $B(r)$, holomorphic in $v$ and $\overline{w}$, and satisfies the interior estimate in Corollary~\ref{co:WintdK}. 
Applying Corollary~\ref{co:WintdK}, with $\Omega = B(r)$ and $E$ being the closure $\overline{B(r/2)}$ of $B(r/2)$, yields
  \[
     |\p_v^{\alpha} \p_{\overline{v}}^{\beta} b (v, v) | \le \frac{\alpha! \beta! C(n)}{r^{2n+|\alpha| + |\beta|}}, \quad \textup{for all $v \in \overline{B(r/2)}$}.
  \]
  On the other hand, by the hypothesis $\mathfrak{B}(P, P) /\omega^n (P) \ge c_0$ for all $P \in M$; hence, 
  \[
     b (v, v) \ge c_0 \det (g_{i\bar{j}}) \ge c_0 C^{-n} > 0,
  \]
  where \eqref{eq:bgC} is used.
  Write $\omega_{\mathfrak{B}} = (\sqrt{-1}/2)g_{\mathfrak{B}, i\bar{j}} dv^i \wedge d \bar{v}^j$. Then,
  \[
     g_{\mathfrak{B}, i\bar{j}} = b^{-1} \p_{v^i} \p_{\overline{v}^j} b - b^{-2} \p_{v^i} b \, \p_{\overline{v}^j} b
  \]
  satisfies that
  \begin{equation} \label{eq:BKVp}
      (g_{\mathfrak{B}, {i\bar{j}}}) \le \frac{C(c_0, n)}{r^{2n+2}} (\delta_{ij}) \le \frac{C(n, c_0, \kappa_1, \kappa_2)}{r^{2n+2}} (g_{i\bar{j}}),
  \end{equation}
  by \eqref{eq:bgC} again, and that
  \[
     \Big| \p_v^{\mu} \p_{\overline{v}}^{\nu} g_{\mathfrak{B}, i\bar{j}} \Big| \le \frac{C(n, c_0)}{r^{2n+ 2+ |\mu| + |\nu|}}.
  \]
  This proves that $\omega_{\mathfrak{B}}$ has bounded geometry. The desired inequality \eqref{eq:conGWB} (or equivalently, $\textup{tr}_{\omega} \omega_{\mathfrak{B}} \le C$) follows from \eqref{eq:BKVp}.
\end{proof}
Note that the hypothesis $\mathfrak{B} \ge c_0 \omega^n$ in Lemma~\ref{le:WBeM} is guaranteed  by the left inequality in Theorem~\ref{th:GrWuH} \eqref{eq:qsBkf}, which is, in fact, implicitly contained in \cite[p. 248, line -4]{Siu-Yau:1977}. Thus,
Theorem~\ref{th:WY-B} 
follows from the left inequality of \eqref{eq:qsBkf} and Lemma~\ref{le:WBeM}.
\begin{remark}
A consequence of Theorem~\ref{th:WY-B} is the following technical fact on the $L^2$-estimate, originally proposed (conjectured) by \cite[p. 145]{Greene-Wu:1979} to show Conjecture~\ref{con:GW-B}.  
Fix arbitrary $x \in M$ and $\eta \in T'_x M$.  
For any $\varphi \in \mathcal{H}$ with $\varphi(x) = 0$, define
\[
   \eta (\varphi ) = \eta (f),
\]
where $\varphi$ is locally represented by $f(z) dz^1\wedge \cdots \wedge dz^n$ near $x$. It is well-defined. Denote
\begin{align*}
   \mathcal{E}_{\eta} (x)
   & = \{\varphi \in \mathcal{H}; \varphi(x) = 0, \eta (\varphi )  = 1\}.
\end{align*}
\begin{coro}\label{co:GW-L2}
Let $(M, \omega)$ be a simply-connected complete K\"ahler manifold whose sectional curvature is bounded between two negative constants $-B$ and $-A$. Then, there exists a constant $C>0$ depending only on $\dim M$, $A$, and $B$, such that 
\[
    \min_{\varphi \in \mathcal{E}_{\eta}(x)} \| \varphi \| \ge C, \quad \textup{for any $x \in M$, $\eta \in T'_x M$}.
\]
\end{coro}
\noindent
Corollary~\ref{co:GW-L2} follows immediately from Lemma 8.17 (A) and Lemma 8.19 in \cite{Greene-Wu:1979} and Theorem~\ref{th:WY-B}.
\end{remark}
\begin{remark}
Theorem~\ref{th:WY-B} can also be compared with a different direction, proposed by the second author, concerning the asymptotic behavior of the Bergman metric on the higher multiple $mK_M$ of the canonical bundle for large $m$. The difference lies not only in the fact that $M$ is noncompact here, but also the situation that one has to consider all terms for the case $m = 1$, rather than the leading order terms for the case $m \to +\infty$. 
\end{remark}

\section{K\"ahler-Einstein metric and holomorphic curvature}
The goal of  this section is to prove Theorem~\ref{th:WYc1}. We shall use the continuity method (Lemma~\ref{le:KEqc}). Theorem~\ref{th:WYc1} follows immediately from Lemma~\ref{le:Shi} and Lemma~\ref{le:KEqc}.

The proof of Lemma~\ref{le:KEqc} differs from that of Cheng-Yau~\cite{Cheng-Yau:1980} and others mainly in the complex Monge-Amp\`ere type equation. The equation used here is inspired by the authors' work~\cite{Wu-Yau:2015}. This new equation is well adapted to the negative holomorphic sectional curvature and the Schwarz type lemma.

As in Cheng-Yau~\cite{Cheng-Yau:1980}, we define the H\"older space $\mathcal{C}^{k,\alpha}(M)$ based on the quasi-coordinates. Let $(M, \omega)$ be a complete K\"ahler manifold of quasi-bounded geometry, and let $\{V_j, \psi_j\}_{j = 1}^{\infty}$ be a family of quasi-coordinate chats in $M$ such that 
\[
   M = \bigcup_{j \ge 1} \psi_j (V_j).
\]
Let $k \in \mathbb{Z}_{\ge 0}$ and $0 < \alpha <1$. For a smooth function $f$ on $M$, define
\[
   | f |_{\mathcal{C}^{k,\alpha}(M)} = \sup_{j \ge 1} \Big( | \psi_j^* f |_{C^{k,\alpha}(V_j)} \Big),
\]
where $|\cdot |_{C^{k,\alpha}(V_j)}$ is the usual H\"older norm on $V_j \subset \mathbb{C}^n$. 
Then, we define $\mathcal{C}^{k,\alpha}(M)$ to be the completion of $\{f \in C^{\infty}(M) ; | f |_{\mathcal{C}^{k,\alpha}(M)} < +\infty\}$ with respect to $|\cdot |_{\mathcal{C}^{k,\alpha}(M)}$.
\begin{lemma}\label{le:ChengYau}
Let $(M, \omega)$ be an $n$-dimensional complete K\"ahler manifold of quasi-bounded geometry, and let $\mathcal{C}^{k,\alpha}(M)$ be an associated H\"older space. For any function $f \in \mathcal{C}^{k,\alpha}(M)$, there exists a unique solution $u \in \mathcal{C}^{k+2,\alpha}(M)$ satisfying
\[
   \left\{ \begin{aligned}
     & (\omega + dd^c u)^n  = e^{u + f} \omega^n \\
     & C^{-1} \omega \le \omega + dd^c u \le C \omega
   \end{aligned} 
   \right.
\]
on $M$. Here the constant $C>1$ depends only on $\inf_M f$, $\sup_M f$, $\inf_M (\Delta_{\omega} f)$, $n$, and $\omega$, in which $\Delta_{\omega} f$ denotes the Laplacian of $f$ with respect to $\omega$.
\end{lemma}
The proof of Lemma~\ref{le:ChengYau} follows from \cite[p. 524, Theorem 4.4]{Cheng-Yau:1980}, with their bounded geometry replaced by the quasi-bounded geometry, which is used in the openness argument, and the bootstrap argument from the third order estimate to $\mathcal{C}^{k,\alpha}(M)$ estimate.

\begin{lemma} \label{le:KEqc}
Let $(M, \omega)$ be an $n$-dimensional complete K\"ahler manifold such that
\[
   H(\omega) \le - \kappa_1 < 0
\]
for some constant $\kappa_1 > 0$.
Assume that for each integer $q \ge 0$, the curvature tensor $R_{\textup{m}}$ of $\omega$ satisfying
\begin{equation} \label{eq:Rmbdd}
   \sup_{x \in M} | \nabla^q R_{\textup{m}} | \le B_q
\end{equation}
for some constant $B_q>0$, where $\nabla^q$ denotes the $q$th covariant derivative with respect to $\omega$. 
Then, there exists a smooth function $u$ on $M$ such that $\omega_{\textup{KE}} \equiv dd^c \log \omega^n + dd^c u$ is the unique K\"ahler-Einstein metric with Ricci curvature equal to $-1$, and satisfies
\begin{equation} \label{eq:isoKE}
   C^{-1} \omega \le \omega_{\textup{KE}} \le C \omega,
\end{equation}
where the constant $C > 0$ depends only on $n$ and $\omega$. Furthermore, the curvature tensor $R_{\textup{m}, \textup{KE}}$ of $\omega_{\textup{KE}}$ and its $q$th covariant derivative satisfies
\[
   \sup_{x \in M} |\nabla_{\textup{KE}}^q R_{\textup{m}, \textup{KE}} | \le C_{q},
\]
for some constant $C_q$ depending only on $n$, and $B_0, \ldots, B_q$. 
\end{lemma}
\begin{proof}
By hypothesis \eqref{eq:Rmbdd} and Theorem~\ref{th:WYqc}, the complete manifold $(M, \omega)$ has quasi-bounded geometry. Denote by $\mathcal{C}^{k,\alpha}(M)$ the associated H\"older space, $k \ge 0$, $0 < \alpha <1$.

Consider the Monge-Amp\`ere equation
\[ 
  \left\{ \begin{aligned}
   (t \omega + dd^c \log \omega^n + dd^c u)^n & = e^{u} \omega^n, \\
    c_t^{-1} \omega \le t \omega + dd^c \log \omega^n + dd^c u & \le c_t \, \omega,
    \end{aligned} \right. \tag*{$($MA$)_t$} \label{eq:ncpMA}
\]
on $M$ with $t > 0$, where the constant $c_t>1$ may depend on $t$.
First, we \textbf{claim} that for a sufficiently large $t$, \ref{eq:ncpMA} has a smooth solution $u$ such that
\begin{equation} \label{eq:qsw_t}
   C^{-1} \omega \le t \omega + dd^c \log \omega^n + dd^c u \le C \omega \quad \textup{on $M$},
\end{equation}
where $C>0$ is a constant depending only on $n$ and $\omega$. To see this, note that $-dd^c \log \omega^n$ is precisely the Ricci curvature of $\omega$. By \eqref{eq:Rmbdd} the curvature tensor of $\omega$ is bounded; then, for an arbitrary  $t_1 > \sqrt{n} B_0$,  
\[
   t_1 \omega > - dd^c \log \omega^n \quad \textup{on $M$}.
\]
It follows that
\[
   t \omega + dd^c \log \omega^n > t_1 \omega \quad \textup{for all $t \ge 2 t_1 > 0$}.
\]
This implies that $t \omega + dd^c \log \omega^n$ defines complete K\"ahler metric on $M$; moreover, since $\omega$ is of quasi-bounded geometry, so is $t \omega + dd^c \log \omega^n$ for $t \ge 2t_1$. In particular,
\[
   F = \log \frac{\omega^n}{(t \omega + dd^c \log \omega^n)^n} \in \mathcal{C}^{k,\alpha}(M), \quad \textup{for all $k \ge 0$, $0 < \alpha <1$.}
\]
It then follows from Lemma~\ref{le:ChengYau} that for $t \ge 2t_1$, equation
\[
   (t \omega + dd^c \log \omega^n + dd^c u)^n = e^{u+F} (t \omega + dd^c \log \omega^n)^n
\]
admits a solution $u \in \mathcal{C}^{k+2,\alpha}(M)$ for all $k\ge 0$ and $0 < \alpha <1$ and satisfies \eqref{eq:qsw_t}. 
This proves the claim.

Let
\[
  T = \{ t \in [0, 2 t_1] ; \; \textup{system \ref{eq:ncpMA} admits a solution $u \in \mathcal{C}^{k+2,\alpha}(M)$}\}.
\]
Then $T$ is nonempty, since $2 t_1 \in T$. 
We would like to show $T$ is open in $[0, 2t_1]$.
Let $t_0 \in T$ with $u_{t_0} \in \mathcal{C}^{k+2,\alpha}(M)$ satisfying $($MA$)_{t_0}$. The linearization of the operator
  \[
     \mathcal{M} (t, v) = \log \frac{(t \omega + dd^c \log \omega^n + dd^c v)^n}{\omega^n} - v
  \]
with respect to $v$ at $t = t_0$, $v = u_{t_0}$ is given by
\[
     \mathcal{M}_{u_{t_0}} (t_0, u_{t_0}) h = \left. \frac{d}{ds} \mathcal{M}(t_0, u_{t_0} + s h) \right|_{s = 0} = (\Delta_{t_0} - 1) h.
\]
Here $\Delta_{t_0}$ denotes the Laplacian with respect to metric $\omega_{t_0} \equiv t_0 \omega + dd^c \log \omega^n + dd^c u_{t_0}$. 
Note that $c_{t_0}^{-1} \omega \le \omega_{t_0} \le c_{t_0} \omega$. In particular, $\omega_{t_0}$ is complete. Furthermore, $\omega_{t_0}$ has quasi-bounded geometry up to order $(k,\alpha)$, that is, $\omega_{t_0}$ has quasi-coordinates satisfying \eqref{eq:simEu}, and \eqref{eq:bddgk} with the norm $|\cdot|_{C^l(U)}$ replaced by $|\cdot |_{C^{k,\alpha}(U)}$. Then, $\Delta_{t_0} - 1 : \mathcal{C}^{k+2,\alpha}(M) \to \mathcal{C}^{k,\alpha}(M)$ is a linear isomorphism, which follows from the same process as that in \cite[pp. 520--521]{Cheng-Yau:1980}, with their bounded geometry replaced by the quasi-bounded geometry. Thus, $T$ is open, by the standard implicit function theorem.

To show $T$ is closed, we shall derive the a priori estimates. Applying the arithmetic-geometry mean inequality to the equation in \ref{eq:ncpMA} yields
\[
   n e^{u/n} \le  n t - s_{\omega} + \Delta_{\omega} u \le C + \Delta_{\omega} u,
\]
where $s_{\omega} \equiv - \textup{tr}_{\omega} dd^c \log \omega^n$ is precisely the scalar curvature of $\omega$. Henceforth, we denote by $C$ and $C_j$ generic positive constants depending only on $n$ and $\omega$. Applying the second author's generalized maximum principle (see, for example, \cite[Proposition 1.6]{Cheng-Yau:1980}) yields
\begin{equation} \label{eq:C0}
   \sup_M u \le C.
\end{equation}
Next, observe that \ref{eq:ncpMA} implies
\begin{equation} \label{eq:Ricmt}
   \textup{Ric} (\omega_t) = - dd^c \log \omega^n_t = - \omega_t + t \omega,
\end{equation}
where $\omega_t \equiv t \omega + dd^c \log \omega^n + dd^c u>0$.
Applying \cite[Proposition 9]{Wu-Yau:2015} with $\omega' = \omega_t$ yields
\[
   \Delta' \log S \ge \Big[ \frac{(n+1)\kappa_1}{2n} + \frac{t}{n} \Big] S - 1,
\]
where $S = \textup{tr}_{\omega_t} \omega$. Again by the second author's generalized maximum principle, 
\begin{equation} \label{eq:C2}
    \sup_M S \le \frac{2n}{(n+1)\kappa_1}.
\end{equation}
Combining \eqref{eq:C0} and \eqref{eq:C2} yields the estimates of $u$ up to the complex second order (cf. \cite{Wu-Yau:2015, Wu-Yau:2016:quasi}). In fact, by \eqref{eq:C2}, 
\[
   e^{-\frac{u}{n}} = \Big(\frac{\omega^n}{\omega_t^n}\Big)^{\frac{1}{n}} \le \frac{S}{n}  \le \frac{2}{(n+1)\kappa_1}.
\]
This implies
\[
    \inf_M u \ge - n \log \frac{(n+1)\kappa_1}{2}.
\]
Moreover, by \eqref{eq:C0} we have $\sup (\omega_t^n / \omega^n) \le C$. This together with \eqref{eq:C2} implies 
\[
   \textup{tr}_{\omega} \omega_t \le  n \Big(\frac{S}{n}\Big)^{n-1} \Big(\frac{\omega_t^n}{\omega^n}\Big)  \le C.
\]
Hence, $\Delta_{\omega} u \le C$ and
\begin{equation} \label{eq:clqsw_t}
    \frac{(n+1)\kappa_1}{2n} \omega \le \omega_t \le (\textup{tr}_{\omega} \omega_t) \omega \le C \omega.
\end{equation}
One can apply \cite[p. 360, 403--406]{Yau:1978:calabi} to the third order term 
\[
   \Xi \equiv g'_{i\bar{j};k} g'_{\bar{r}s;\bar{t}} g'^{i\bar{r}} g'^{s \bar{j}} g'^{k \bar{t}}
\]
 to get  
\[
   \Delta' (\Xi + C \Delta_{\omega} u ) \ge C_1 (\Xi + C \Delta_{\omega} u ) - C_2,
 \]
 where $g'_{i\bar{j}}$ is the metric component of $\omega_t$ and the subscript $;k$ in $g'_{i\bar{j};k}$ denotes the covariant derivative along $\p / \p z^k$ with respect to $\omega$.
 Thus, $\sup_M \Xi \le C$ by the second author's generalized maximum principle. Now applying the standard bootstrap argument (see \cite[p. 363]{Yau:1978:calabi}) to the equation in \ref{eq:ncpMA} with the quasi-local coordinate charts yields $\| u \|_{\mathcal{C}^{k+2,\alpha}(M)} \le C$. The desired closedness of $T$ then follows immediately from the standard Ascoli-Arzel\`a theorem and \eqref{eq:clqsw_t}.
 
 Hence, we have proven $t = 0 \in T$ with $u \in \mathcal{C}^{k,\alpha}(M)$. Then, formula \eqref{eq:Ricmt} tells us that $dd^c \log \omega^n + dd^c u$ is the K\"ahler-Einstein metric. The uniform equivalence \eqref{eq:isoKE} and boundedness of covariant derivatives of its curvature tensor follow immediately from the above uniform estimates on $u$. The uniqueness of complete K\"ahler-Einstein metric of negative curvature follows immediately from the second author's Schwarz Lemma (\cite[Theorem 3]{Yau:1978:schwarz}; see also \cite[p. 408, Lemma 3.3]{Wu:2008}).  
\end{proof}

\bibliographystyle{alpha}
\bibliography{../../../Bib/DWu}

\end{document}